\documentclass{amsart}

\usepackage{amsmath}
\usepackage{amssymb}
\usepackage{amsthm}
\usepackage{amsfonts,comment}
\usepackage{mathtools}
\usepackage{enumitem}[shortlabels]
\usepackage{tikz}
\usepackage{tikz-cd}
\usepackage{mleftright}
\usepackage[T1]{fontenc}
\usepackage{stmaryrd}

\theoremstyle{definition}
\newtheorem{definition}{Definition}[section]
\newtheorem{example}[definition]{Example}
\newtheorem{assumption}[definition]{Assumption}
\newtheorem*{notation}{Notation}
\newtheorem*{question}{Question}

\theoremstyle{plain}
\newtheorem{proposition}[definition]{Proposition}
\newtheorem{theorem}[definition]{Theorem}
\newtheorem{corollary}[definition]{Corollary}
\newtheorem{lemma}[definition]{Lemma}
\newtheorem*{theorem*}{Theorem}

\theoremstyle{remark}
\newtheorem{remark}[definition]{Remark}

\newcommand{\id}{\textrm{\normalfont id}}

\newcommand{\R}{\mathbb{R}}
\newcommand{\C}{\mathbb{C}}

\newcommand{\D}{\mathcal{D}}

\newcommand{\catname}[1]{{\normalfont\textsf{#1}}}
\newcommand{\Alg}[1]{\catname{Alg}_{#1}}

\DeclareMathOperator{\DStr}{\mathcal{D}\text{\normalfont -Str}}
\DeclareMathOperator{\DoneStr}{\mathcal{D}_1\text{\normalfont -Str}}
\DeclareMathOperator{\DtwoStr}{\mathcal{D}_2\text{\normalfont -Str}}
\DeclareMathOperator{\DonetwoStr}{\mathcal{D}_1 \mathcal{D}_2\text{\normalfont -Str}}
\DeclareMathOperator{\DtwooneStr}{\mathcal{D}_2 \mathcal{D}_1\text{\normalfont -Str}}

\title[Weil descent of algebras with free operators]{The Weil descent functor in the category of algebras with free operators}
\author{Shezad Mohamed}
\address{Shezad Mohamed, Department of Mathematics, University of Manchester, Oxford Road, Manchester, United Kingdom M13 9PL}
\email{shezad.mohamed@manchester.ac.uk}

\date{\today}
\thanks{{\em Acknowledgements}: This research was supported by a University of Manchester Research Scholar Award.}

\begin{document}

\begin{abstract}
We prove that there exists a version of Weil descent, or Weil restriction, in the category of $\D$-algebras. The objects of this category are $k$-algebras $R$ equipped with a homomorphism $e \colon R \to R \otimes_k \D$ for some fixed field $k$ and finite-dimensional $k$-algebra $\D$. We do this under a mild assumption on the so-called associated endomorphisms. In particular, this yields the existence of the Weil descent functor in the category of difference algebras, which, to our knowledge, does not appear elsewhere.
\end{abstract}

\maketitle

\section{Introduction}

It is well known that if $L/K$ is a finite field extension, then extension of scalars, considered as a functor $F \colon \Alg{K} \to \Alg{L}$, has a left adjoint, $W$. That is, $W$ is a functor $\Alg{L} \to \Alg{K}$ such that there is a natural bijection $\text{Hom}_{\Alg{K}}(W(C),R) \to \text{Hom}_{\Alg{L}}(C, F(R))$. This left adjoint is known as Weil descent or Weil restriction---see, for instance, Section~1.3 of \cite{weil-adeles}. In fact, this result has been generalised by Grothendieck to the case when $K$ is an arbitrary commutative ring and $L$ is a $K$-algebra which, as a $K$-module, is free and finitely generated (see \cite{Grothendieck1958-1960}).

This classical Weil descent has been used in applications to number theory \cite{smart} and algebraic geometry \cite{milne_arithmetic_1972}. It is also fundamental to the construction of prolongation spaces in the sense of Moosa-Scanlon \cite{moosa_scanlon_2010}. Furthermore, since the adjunction gives rise to the natural bijection $\text{\normalfont Hom}_{\Alg{K}}(W(C),K) \to \text{\normalfont Hom}_{\Alg{L}}(C,L)$, when $C$ is the coordinate ring of an affine $L$-variety $V$, we obtain a bijection between the $L$-rational points of $V$ and the $K$-rational points of the variety $\text{Spec}(W(C))$. This fact is used by Pop in \cite{pop96} to show that algebraic extensions of large fields are large.

In \cite{differential_weil_descent} the case of differential algebras is considered. The authors show that the differential base change functor, $F^\delta$, has a left adjoint, which they call the differential Weil descent functor, $W^\delta$. More precisely, they show that if $(A,\partial)$ is a differential ring and $(B,d)$ an $(A,\partial)$-algebra, where $B$ is finite and free as an $A$-module, then for any $(B,d)$-algebra $(D,\delta)$, there exists a unique derivation $\delta^W$ on $W(D)$ making the unit of the classical adjunction into a differential ring homomorphism. The authors then use this result in a similar way to Pop to show that algebraic extensions of differentially large fields are again differentially large (see \cite{leon_sanchez_differentially_2020}).

It is natural, then, to explore whether the difference base change functor, $F^\sigma$---here difference rings are rings equipped with an endomorphism---also has a left adjoint. In general, it does not. Let $A$ be a commutative unital ring and consider the case when $B = A[\varepsilon]/(\varepsilon^2)$ for an indeterminate $\varepsilon$. Let $\tau \colon B \to B$ be given by $\tau(a+b\varepsilon) = a$ so that $(A,\id_A) \leq (B, \tau)$. Let $R = B[x]$ and $\rho \colon R \to R$ the unique endomorphism extending $\tau$ and sending $x \mapsto \varepsilon$. If $F^\sigma = F^\sigma_{B/A}$ had a left adjoint $W^\sigma$, then the unit of this adjunction at $R$
\[
\eta^\sigma_R \colon R \to F^\sigma W^\sigma(R)
\]
would be a difference ring homomorphism. In particular
\begin{equation}\tag{$\dagger$}
    \eta^\sigma_R(\rho(x)) = (\theta \otimes \tau)(\eta^\sigma_R(x))
\end{equation}
where $\theta$ is the endomorphism of $W^\sigma(R)$. Let $\lambda_1$ and $\lambda_2$ be the coordinate projections with respect to the $A$-basis $\{1, \varepsilon \}$ of $B$. Then, equation $(\dagger)$ translates to
\[
\mleft[
\begin{array}{c}
  \lambda_1(\rho(x)) \\
  \lambda_2(\rho(x))
\end{array}
\mright]
=
\mleft[
\begin{array}{c c}
    \lambda_1(\tau(1)) & \lambda_1(\tau(\varepsilon)) \\
    \lambda_2(\tau(1)) & \lambda_2(\tau(\varepsilon))
\end{array}
\mright]
\mleft[
\begin{array}{c}
    \theta(\lambda_1(\eta^\sigma_R(x))) \\
    \theta(\lambda_2(\eta^\sigma_R(x)))
\end{array}
\mright]
\]
See Lemma~\ref{technical-lemma} for details on this. Using the facts $\rho(x) = \varepsilon$, $\tau(1) = 1$, and $\tau(\varepsilon) = 0$, the above yields 
\[
\mleft[
\begin{array}{c}
  0 \\
  1
\end{array}
\mright]
=
\mleft[
\begin{array}{c c}
    1 & 0 \\
    0 & 0
\end{array}
\mright]
\mleft[
\begin{array}{c}
    \theta(\lambda_1(W^\sigma_R(x))) \\
    \theta(\lambda_2(W^\sigma_R(x)))
\end{array}
\mright]
\]
which is clearly inconsistent. Hence, equation $(\dagger)$ cannot hold, and the left adjoint $W^\sigma$ cannot exist. The issue here is that the $2 \times 2$ matrix on the right-hand side that we associate to $(B,\tau)$ is not invertible. In this case we say that $\tau$ does not have invertible matrix. We will see in the course of Section~5 that $\tau$ having invertible matrix is sufficient for a left adjoint to exist, and, in Section~6, that in the case when $A$ is a field, it is also necessary.

\begin{theorem*}
Let $(A,\sigma)$ be a difference ring and $(B,\tau)$ a difference $(A, \sigma)$-algebra where $B$ is finitely generated and free as an $A$-module, and $\tau$ has invertible matrix. If $(C,\rho)$ is a difference $(B,\tau)$-algebra, then there is a unique endomorphism $\rho^W$ on the classical Weil restriction, $W(C)$, making $(W(C),\rho^W)$ into a difference $(A, \sigma)$-algebra and the unit of the classical adjunction $\eta_C \colon C \to W(C) \otimes_A B$ into a difference ring homomorphism $(C,\rho) \to (W(C) \otimes_A B, \rho^W \otimes \tau)$. The assignment $(C, \rho) \mapsto (W(C), \rho^W)$ is the left adjoint to the difference base change functor.
\end{theorem*}

One might initially think to define $\rho^W = W(\rho)$. However, while $\rho$ is a ring endomorphism, it is not in general a $B$-algebra homomorphism. Thus $W(\rho)$ does not exist. We will see in Section~5 that this idea works once we ``twist'' the $B$-algebra structure on $C$.

In fact, we prove our results in the more general framework of $\D$-operators, originally introduced in \cite{moosa_scanlon_2010}, of which endomorphisms are a special case. We refer the reader to Section~3 for details on $\D$-rings, but give a brief presentation here. For a fixed field $k$ and a fixed finite-dimensional $k$-algebra $\D$, by a $\D$-ring we mean a $k$-algebra $R$ equipped with a $k$-algebra homomorphism $e \colon R \to R \otimes_k \D$. If we coordinatise $e$ with respect to a basis of $\D$, we obtain a sequence of additive operators $R \to R$---the coordinate maps of $e$ with respect to the chosen basis. These are often called free operators in the literature (for instance in \cite{ozlem_2018} and \cite{moosa_scanlon_2013}). Differential rings fit into this framework of $\D$-rings since $\delta$ is a derivation on $R$ if and only if the map
\begin{align*}
    R &\to R[\varepsilon]/(\varepsilon^2) \\
    r &\mapsto r + \delta(r) \, \varepsilon
\end{align*}
is a homomorphism, as do difference rings (trivially) since $\sigma$ is an endomorphism of $R$ if and only if the map
\begin{align*}
    R & \to R \\
    r & \mapsto \sigma(r)
\end{align*}
is a homomorphism. For further details and examples see Example~\ref{D-ring-examples}. 

Now, let $(A,e)$ be a $\D$-ring and $(B,f)$ an $(A,e)$-algebra where $B$ is finite and free as an $A$-module. The $\D$-structure on $B$ has associated endomorphisms (on $B$, see Definition~\ref{associated-endomorphisms}) and, as in the difference case, if the associated endomorphisms of $(B,f)$ do not have invertible matrix, then the left adjoint to the $\D$-base change functor (see Corollary~\ref{D-base-change}) does not generally exist. Nonetheless, our main result states that this is indeed the main obstacle: if the associated endomorphisms of $(B,f)$ have invertible matrix, then the $\D$-base change functor has a left adjoint, and if $A$ is a field, this condition is necessary. See Theorem~\ref{main-theorem} and Corollary~\ref{corollary-inv-matrix-is-necessary}.

Besides being of independent interest, our work here on the $\D$-Weil descent is partly motivated by a model-theoretic study of a uniform companion for theories of large $\D$-fields, called $\text{UC}_\D$, in the spirit of the uniform companion of Tressl for theories of large differential fields, UC, established in \cite{tressl_uniform_2005}. In a follow-up paper we show that algebraic extensions of models of $\text{UC}_\D$ which are large as a field are again models of $\text{UC}_\D$---assuming that $\D$ is a local ring and the associated endomorphism is trivial. We can then conclude that the algebraic closure of such a $\D$-field is a model of the theory $\D$-$\text{CF}_0$ of Moosa-Scanlon \cite{moosa_scanlon_2013}. The argument we have in mind to prove these results relies on the existence of a $\D$-Weil descent---the content of this paper. In the differential case this argument is done in \cite{leon_sanchez_differentially_2020}.

The paper is organised as follows: Section~2 gives a brief overview of the classical Weil descent and left adjoints in general. Section~3 then gives an introduction to $\D$-rings and $\D$-algebras, as well as some algebraic notions about them. It also recalls the definition of the $\D$-base change functor. Section~4 examines sufficient conditions for the $\D$-Weil descent to exist, and Section~5 contains the proof of our main result. Section~6 contains a partial converse to the main theorem, as well as some results on properties preserved under taking the Weil descent. In particular, it will explain how to apply our results in the case of several commuting endomorphisms and derivations. Finally, the appendix shows how to explicitly construct the $\D$-Weil descent, mirroring the construction in the classical case.

The author would like to thank Omar Le\'{o}n S\'{a}nchez for several helpful discussions on the content of this paper and for support while writing it, as well as the anonymous referee for their helpful comments and questions, which prompted a new, more natural proof of the main theorem. The original proof is sketched in the appendix.

Throughout this paper we assume that all rings and algebras are commutative and unital, and that ring and algebra homomorphisms preserve the unit.

\section{Preliminaries on the classical Weil descent}

In this section we briefly go over the details of the construction of the classical Weil descent. We will not give proofs, but the reader can consult \cite[\S7.6]{neron_1990} and \cite[\S2]{moosa_scanlon_2010} for further details. Our approach is modelled after \cite{differential_weil_descent}, so the reader can also consult there for a more in-depth explanation.

Let $A$ be a ring, and $B$ an $A$-algebra. For any $A$-algebra $R$ we can form the base change of $R$ to $B$, namely $R \otimes_A B$, where the $B$-algebra structure is given by $b \mapsto 1 \otimes b$. This base change naturally extends to a functor $F \colon \Alg{A} \to \Alg{B}$ where we set $F(\phi) = \phi \otimes \id_B$. If we let $G \colon \Alg{B} \to \Alg{A}$ be the scalar restriction functor, where $G(C)$ is the $A$-algebra given by composing $A \to B \to C$, then $G$ is right adjoint to $F$. More importantly, if $B$ is free of finite rank as an $A$-module, then $F$ has a left adjoint: Weil restriction $W \colon \Alg{B} \to \Alg{A}$.

We state the following useful fact about adjunctions from Theorem~2 and Corollaries~1 and 2 of \cite{mac_lane_categories_1978}. 

\begin{theorem}\label{adjunction-conditions}
Let $F \colon \mathcal{X} \to \mathcal{Y}$ be a functor, and suppose that for each $C \in \mathcal{Y}$, there is some $W(C) \in \mathcal{X}$ and $\eta_C \colon C \to F(W(C))$ in $\mathcal{Y}$ such that the assignment $g \mapsto F(g) \circ \eta_C$ is a bijection $\text{\normalfont Hom}_\mathcal{X}(W(C),R) \to \text{\normalfont Hom}_\mathcal{Y}(C,F(R))$. Then, $W$ extends to a functor $\mathcal{Y} \to \mathcal{X}$ which is left adjoint to $F$. The unit of this adjunction is given by $\eta_C$.

In particular, for a morphism $f \colon C \to C'$ in $\mathcal{Y}$, $W(f)$ is defined to be the unique morphism $g \colon W(C) \to W(C')$ such that $F(g) \circ \eta_C = \eta_{C'} \circ f$.
\end{theorem}

This fact will allow us to construct the left adjoint using only the data of its object map and unit. This fact is also the method of proof for the differential Weil descent in \cite{differential_weil_descent}.

For the convenience of the reader, we now briefly explain the situation in the classical setup. Let $b_1, \ldots, b_r$ be an $A$-basis of $B$. For each $i=1, \ldots, r$, let $\lambda_i \colon B \to A$ be the $A$-module homomorphism with $\lambda_i \left( \sum_{j=1}^r a_j b_j \right) = a_i$. If $R$ is an $A$-algebra, we consider the base change of $\lambda_i$ to $R$---the $R$-module homomorphism $\id_R \otimes \lambda_i \colon R \otimes_A B \to R$. Note that $\id_R \otimes \lambda_i$ simply picks out the coefficient of the basis element $1 \otimes b_i$. We will write $\lambda_i$ for $\id_R \otimes \lambda_i$ throughout, but it will be clear from context which we mean.

Now let $T$ be a set of indeterminates, and define 
\[
W(B[T]) = A[T]^{\otimes r} = A[T] \otimes_A A[T] \otimes_A \ldots \otimes_A A[T]
\]
For each $i$ and $t \in T$, let $t(i) = 1 \otimes \ldots \otimes 1 \otimes t \otimes 1 \otimes \ldots \otimes 1$, where the $t$ occurs in the $i$th position. We also let $\eta_{B[T]}$ be the unique $B$-algebra homomorphism
\begin{align*}
    \eta_{B[T]} \colon B[T] & \to F(W(B[T])) \\
    t & \mapsto \sum_{i=1}^r t(i) \otimes b_i
\end{align*}

These choices make the following map $\tau(B[T],R)$ a bijection for each $A$-algebra $R$:
\begin{align*}
\text{Hom}_{\Alg{A}}(A[T]^{\otimes r}, R) & \to \text{Hom}_{\Alg{B}}(B[T], R \otimes_A B) \\
\phi & \mapsto F(\phi) \circ \eta_{B[T]}
\end{align*}
where the compositional inverse is defined as follows. For $\psi \colon B[T] \to R \otimes_A B$ a $B$-algebra homomorphism, let $\phi$ be the unique $A$-algebra homomorphism with $\phi(t(i)) = \lambda_i(\psi(t))$.

Now let $C$ be a $B$-algebra, and take a surjective $B$-algebra homomorphism $\pi_C \colon B[T] \to C$ for some set of indeterminates $T$. Let $I_C$ be the ideal of $W(B[T])$ generated by all the $\lambda_i ( \eta_{B[T]}(f))$ where $f$ ranges over $\ker \pi_C$. Now define $W(C) = W(B[T]) / I_C$ and $W(\pi_C) \colon W(B[T]) \to W(C)$ as the residue map.

Then we induce a map $\tau(C,R) \colon \text{Hom}_{\Alg{A}}(W(C), R) \to \text{Hom}_{\Alg{B}}(C, F(R))$ which makes the following diagram commute:
\[
\begin{tikzcd}
{\text{Hom}_{\Alg{A}}(W(C),R)} \arrow[dd, "\_ \circ W(\pi_C)"] \arrow[rr, "{\tau(C,R)}"] &  & {\text{Hom}_{\Alg{B}}(C,F(R))} \arrow[dd, "\_ \circ \pi_C"] \\
                                                                     &  &                                         \\
{\text{Hom}_{\Alg{A}}(W(B[T]),R)} \arrow[rr, "{\tau(B[T],R)}"]                           &  & {\text{Hom}_{\Alg{B}}(B[T],F(R))}                          
\end{tikzcd}
\]

Let $\eta_C = \tau(C,W(C))(\id_{W(C)})$, and note that
\[
\eta_C(\pi_C(t)) = \sum_{i=1}^r W(\pi_C)(t(i)) \otimes b_i
\]

From this we see that $\tau(C,R)(\phi) = F(\phi) \circ \eta_C$ and that $\tau(C,R)$ is a bijection, satisfying the conditions of Theorem~\ref{adjunction-conditions}. Then, $W$ is a functor which is left adjoint to $F$ with unit $\eta_C$. This $W$ is the classical Weil descent functor.

\section{$\mathcal{D}$-rings and $\D$-algebra}

In this section we review the basic definitions of $\D$-rings as well as introduce some algebraic notions which we will need later. See \cite{ozlem_2018}, \cite{moosa_scanlon_2010}, and \cite{moosa_scanlon_2013} for more details.

Fix a base field $k$, and let $\D$ be a finite-dimensional $k$-algebra. For the remainder of this paper, we impose the following assumption:

\begin{assumption}\label{D-assumptions}
Since $\D$ is a finite-dimensional $k$-algebra, we may decompose it as a finite product of local, finite-dimensional $k$-algebras (see \cite[Theorem 8.7]{atiyah_introduction_2018}), say $\D = B_1 \times \ldots \times B_t$. We assume that the residue field of each $B_i$ is actually $k$.
\end{assumption}

This is the same assumption imposed by the authors in \cite{moosa_scanlon_2013}.

For any $k$-algebra $R$, we define $\mathcal{D}(R) = R \otimes_k \D$ to be the base change of $\D$ to $R$. Note that $\D(R)$ remains free and finite as an $R$-module. We will often identify a $k$-basis of $\D$ with the corresponding $R$-basis of $\D(R)$. By a slight abuse of notation, we think of $\D$ also as a functor $\Alg{k} \to \Alg{k}$, where for a $k$-algebra homomorphism $\phi \colon R \to S$, $\D(\phi) = \phi \otimes \id_\D$.

\begin{definition}
A $k$-algebra $R$ equipped with a map $e \colon R \to \D(R)$ is a $\D$-ring if $e$ is a $k$-algebra homomorphism. In this case, we call $e$ the $\D$-structure or the $\D$-ring structure on $R$.
\end{definition}

\begin{example}\label{D-ring-examples}
\begin{enumerate}
    \item Take $\D$ to be the algebra of dual numbers, $k[\varepsilon] / (\varepsilon^2)$, with the standard $k$-algebra structure. If $(R,e)$ is a $\D$-ring, let $\sigma$ and $\delta$ be such that $e(r) = \sigma(r) + \delta(r) \varepsilon$. Then, $\sigma$ is a $k$-linear endomorphism of $R$, and $\delta$ is a $k$-linear derivation on $R$ which is twisted by $\sigma$. Indeed, the $k$-linearity of $e$ implies $k$-linearity of $\sigma$ and $\delta$, and multiplicativity implies that 
    \[
    \sigma(rs) + \delta(rs) \varepsilon = \sigma(r)\sigma(s) + (\sigma(r)\delta(s) + \delta(r) \sigma(s)) \varepsilon
    \]
    Note that if a $\D$-ring has $\sigma = \id_R$, then it is a differential $k$-algebra.
    \item Take $\D = k^l$ with the product $k$-algebra structure. If $(R,e)$ is a $\D$-ring, let $e(r) = \sum_i \sigma_i(r) \varepsilon_i$ where $\varepsilon_i$ is the standard basis of $\D$. Then, $\D$-rings are precisely rings with $l$ (not necessarily commuting) $k$-endomorphisms $\sigma_1, \ldots, \sigma_l$.
\end{enumerate}
We refer the reader to \cite{moosa_scanlon_2013} for more examples.
\end{example}

\begin{definition}\label{associated-endomorphisms}
Let $(R,e)$ be a $\D$-ring. Recall Assumption~\ref{D-assumptions} that $\D = B_1 \times \ldots \times B_t$ and $B_i$ is local with residue field $k$. Let $\pi_i \colon \D \to B_i \to k$ be the composition of the projection onto the $i$th coordinate and the residue map. These $k$-algebra homomorphisms lift to $\pi^R_i \colon \D(R) \to R$. Define $\sigma_i = \pi_i^R \circ e$ for $i=1, \ldots, t$. Then, $\sigma_1, \ldots, \sigma_t$ are called the associated endomorphisms of the $\D$-ring $(R,e)$.
\end{definition}

\begin{remark}
If $\D = k^l$ as in Example~\ref{D-ring-examples}(2) above, then the associated endomorphisms of a $\D$-ring $(R,e)$ are just the endomorphisms $\sigma_1, \ldots, \sigma_l$: the coordinate functions of the homomorphism $e$ with respect to the standard basis of $k^l$.
\end{remark}

\begin{remark}
In \cite{moosa_scanlon_2013} the authors impose the condition that the $\D$-ring structure $e$ must be a section to the projection map $\pi_1^R \colon \D(R) \to R$. This forces $\sigma_1$ to be the identity. Under this condition, every $\D$-ring in Example~\ref{D-ring-examples}(1) has $\sigma = \id_R$ and hence may be thought of as a differential $k$-algebra. We do not impose this condition in this paper, but we will show that our results also work in this context in Section~6.1.
\end{remark}

We now specify the morphisms of the categories we are working in. These were defined in Section~3.1 of \cite{moosa_scanlon_ghs}.

\begin{definition}\label{D-maps-def}
If $(R,e)$ and $(S,f)$ are two $\D$-rings, then $\phi \colon (R,e) \to (S,f)$ is a $\D$-homomorphism if it is a $k$-algebra homomorphism and the following diagram commutes:
\[
\begin{tikzcd}
\mathcal{D}(R) \arrow[r, "\mathcal{D}(\phi)"] & \mathcal{D}(S)   \\
R \arrow[r, "\phi"] \arrow[u, "e"]            & S \arrow[u, "s"]
\end{tikzcd}
\]

If $S$ is an $R$-algebra, then we will call $(S,f)$ an $(R,e)$-algebra if the structure map $R \to S$ is a $\D$-homomorphism. If $(S,f)$ and $(T,g)$ are both $(R,e)$-algebras and $\phi \colon S \to T$ is a map between them, then we say that $\phi$ is a $(R,e)$-algebra homomorphism if it is an $R$-algebra homomorphism and a $\D$-homomorphism.
\end{definition}

\begin{remark}
Note that in the context of Example~\ref{D-ring-examples}(1) above, where $\sigma$ is the identity map, a map being a $\D$-homomorphism is equivalent to it being a differential ring homomorphism. In the context of Example~\ref{D-ring-examples}(2), being a $\D$-homomorphism is equivalent to being a difference ring homomorphism for each endomorphism.
\end{remark}

From now on we will denote by $\Alg{(R,e)}$ the category of $(R,e)$-algebras with $(R,e)$-algebra homomorphisms.

\subsection{The tensor product of $\D$-structures}

We now need the correct notion of base change in the context of $\D$-algebras. That is, given a $\D$-ring $(R,e)$ and an $(R,e)$-algebra $(T,g)$, for any $(R,e)$-algebra $(S,f)$ we need a $\D$-ring structure on $S \otimes_R T$ that makes $S \otimes_R T$ into a $(T,g)$-algebra. In \cite{ozlem_2018}, it is proved that there exists a unique $\D$-structure, $f \otimes g$ (called $(\tilde{f}, \tilde{g})$ in \cite{ozlem_2018}), on $S \otimes_R T$ which makes the natural maps $\phi_S \colon S \to S \otimes_R T$ and $\phi_T \colon T \to S \otimes_R T$ into $\D$-homomorphisms. We recall the definition of this structure:

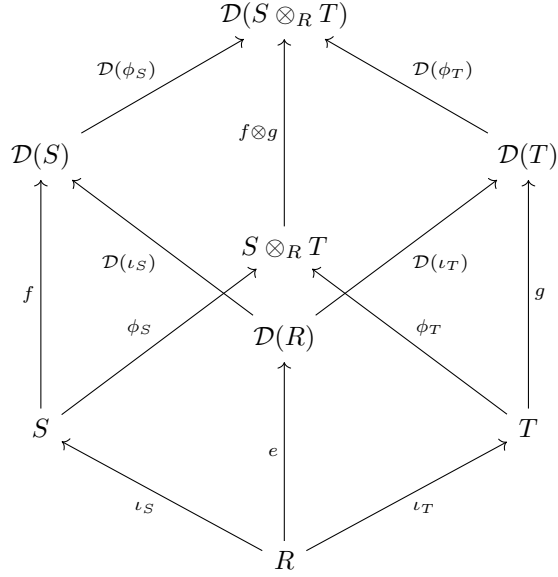
\[
\begin{tikzcd}
                                                   &  & \mathcal{D}(S \otimes_R T)                                                               &  &                                                    \\
                                                   &  &                                                                                          &  &                                                    \\
\mathcal{D}(S) \arrow[rruu, "\mathcal{D}(\phi_S)"] &  &                                                                                          &  & \mathcal{D}(T) \arrow[lluu, "\mathcal{D}(\phi_T)", swap] \\
                                                   &  & S \otimes_R T \arrow[uuu, "f \otimes g"]                                                 &  &                                                    \\
                                                   &  & \mathcal{D}(R) \arrow[lluu, "\mathcal{D}(\iota_S)"] \arrow[rruu, "\mathcal{D}(\iota_T)", swap] &  &                                                    \\
S \arrow[rruu, "\phi_S"] \arrow[uuu, "f"]          &  &                                                                                          &  & T \arrow[lluu, "\phi_T", swap] \arrow[uuu, "g", swap]          \\
                                                   &  &                                                                                          &  &                                                    \\
                                                   &  & R \arrow[lluu, "\iota_S"] \arrow[uuu, "e"] \arrow[rruu, "\iota_T", swap]                       &  &                                                   
\end{tikzcd}
\]

Explicitly, 
\[
(f \otimes g)(s \otimes t) = \left( \D(\phi_S) \circ f(s) \right) \cdot \left( \D(\phi_T) \circ g(t) \right)
\]
where $\cdot$ is the product in $\D(S \otimes_R T)$.

A short computation shows that this agrees with the correct notions of derivations on tensor products: $(\delta \otimes d)(s \otimes t) = \delta(s) \otimes t + s \otimes d(t)$ (see \cite[pg 21]{buium94}), and endomorphisms on tensor products: $(\sigma \otimes \tau) (s \otimes t) = \sigma(s) \otimes \tau(t)$.

In addition, if we have an $(R,e)$-algebra homomorphism $\theta \colon S \to U$, then the map $\theta \otimes \id_T \colon S \otimes_R T \to U \otimes_R T$ is a $(T,g)$-algebra homomorphism. Indeed, we see this from the following diagram:
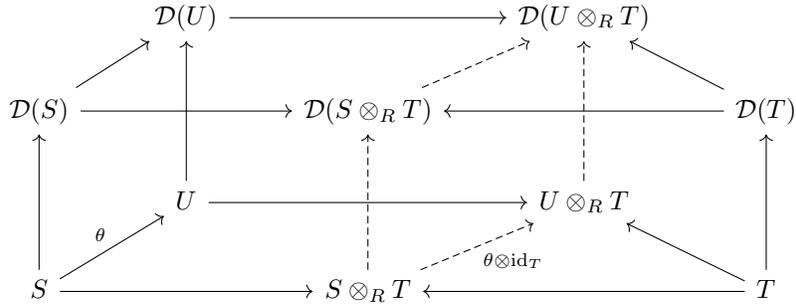
\[
\begin{tikzcd}
                                             & \mathcal{D}(U) \arrow[rr] &                                                                                    & \mathcal{D}(U \otimes_R T)      &                                      \\
\mathcal{D}(S) \arrow[ru] \arrow[rr]         &                           & \mathcal{D}(S \otimes_R T) \arrow[ru, dashed]                                     &                                  & \mathcal{D}(T) \arrow[ll] \arrow[lu] \\
                                             & U \arrow[uu] \arrow[rr]   &                                                                                    & U \otimes_R T \arrow[uu, dashed] &                                      \\
S \arrow[ru, "\theta"] \arrow[uu] \arrow[rr] &                           & S \otimes_R T \arrow[ru, "\theta \otimes \text{id}_T"', dashed] \arrow[uu, dashed] &                                  & T \arrow[lu] \arrow[ll] \arrow[uu]  
\end{tikzcd}
\]

$\theta \otimes \id_T$ is clearly a $T$-algebra homomorphism, so it remains to show the dashed square commutes. Every other face in this diagram commutes by the result above, and since every element of $S \otimes_R T$ is a sum of elements of the form $s \otimes t$, the dashed face also commutes. So we have proved the following:

\begin{corollary}\label{D-base-change}
Let $(R,e)$ be a $\D$-ring, and $(T,g)$ an $(R,e)$-algebra. Define $F^\D \colon \Alg{(R,e)} \to \Alg{(T,g)}$ by $F^\D(S,f) = (S \otimes_R T, f \otimes g)$ on objects, and $F^\D(\theta) = \theta \otimes \id_T$ on morphisms. Then, $F^\D$ is a functor which is just the classical base change functor on the underlying algebras.
\end{corollary}

We finish this section with the following lemma which will be used in Section~6.1. It describes the associated endomorphisms of the $\D$-structure on tensor products.

\begin{lemma}\label{lemma-ass-ends-of-tensor-prod}
Let $(R,e)$ be a $\D$-ring and $(S,f), (T,g) \in \Alg{(R,e)}$. If $(S,f)$ has associated endomorphisms $\sigma_i$ and $(T,g)$ has associated endomorphisms $\tau_i$, then $(S \otimes_R T, f \otimes g)$ has associated endomorphisms $\sigma_i \otimes \tau_i$.
\end{lemma}
\begin{proof}
Using the notation of Definition~\ref{associated-endomorphisms}, we have
\begin{align*}
    \pi_i^{S \otimes_R T} \circ (f \otimes g)(s \otimes t) &= \pi_i^{S \otimes_R T}(\D(\phi_S) \circ f(s)) \cdot \pi_i^{S \otimes_R T}(\D(\phi_T) \circ g(t)) \\
    &= (\pi_i^S \circ f(s) \otimes 1) \cdot (1 \otimes \pi_i^T \circ g(t)) \\
    &= \sigma_i(s) \otimes \tau_i(t)
\end{align*}
\end{proof}

\section{The matrix associated to a free and finite $\D$-ring}
In this section we establish some technical results that will be needed to construct a left adjoint to $F^\D$ in Section~5. We carry forward the notation from the previous section. In particular, $k$ is a field, $\D$ is a finite-dimensional $k$-algebra, and Assumption~\ref{D-assumptions} still holds.

Recall from the example in the introduction that, in general, the difference base change functor had no left adjoint. There, the nonexistence of the left adjoint is due to the fact that the matrix associated to the endomorphism, 

\[
\mleft[
\begin{array}{c c}
    \lambda_1(f(1)) & \lambda_1(f(\varepsilon)) \\
    \lambda_2(f(1)) & \lambda_2(f(\varepsilon))
\end{array}
\mright]
=
\mleft[
\begin{array}{c c}
    1 & 0 \\
    0 & 0
\end{array}
\mright],
\]
is not invertible.

We will show in Section~5 that if the associated matrix is invertible, then we can construct a left adjoint to $F^\D$. The next subsection investigates conditions under which the associated matrix is invertible.

\bigskip

\subsection{The matrix associated to an endomorphism}

As before, let $A$ be a ring and $B$ an $A$-algebra which is finite and free as an $A$-module. We fix a ring endomorphism $\sigma \colon B \to B$ with $\sigma(A) \subseteq A$.

\begin{definition}
For an $A$-basis $b = (b_1, \ldots, b_r)$ of $B$, let $M^\sigma_b$ be the following matrix associated to $\sigma$:
\[
M^\sigma_b = 
\mleft[
\begin{array}{c c c c}
  \lambda_1(\sigma(b_1)) & \lambda_1(\sigma(b_2)) & \cdots & \lambda_1(\sigma(b_r)) \\
  \lambda_2(\sigma(b_1)) & \lambda_2(\sigma(b_2)) & \cdots & \lambda_2(\sigma(b_r)) \\
  \vdots & & \ddots & \\
  \lambda_r(\sigma(b_1)) & \lambda_r(\sigma(b_2)) & \cdots & \lambda_r(\sigma(b_r))
\end{array}
\mright]
\]
where $\lambda_i$ is the $i$th coordinate projection $B \to A$ with respect to the basis $b$. Note that the maps $\lambda_i$ are dependent on the basis $b_1, \ldots, b_r$ and hence will change if the basis changes.
\end{definition}

We will say that $\sigma$ has invertible matrix with respect to the basis $b = (b_1, \ldots, b_r )$ if $M^\sigma_b$ is invertible in $\text{Mat}_{r \times r}(A)$.

\begin{proposition}\label{inv-matrix-equivalences}
The following are equivalent:
\begin{enumerate}
    \item[\normalfont (i)] $\sigma$ has invertible matrix with respect to some $A$-basis of $B$.
    \item[\normalfont (ii)] $\sigma$ has invertible matrix with respect to every $A$-basis of $B$.
    \item[\normalfont (iii)] If $b_1, \ldots, b_r$ is an $A$-basis of $B$, then $\sigma(b_1), \ldots, \sigma(b_r)$ is also an $A$-basis of $B$.
    \item[\normalfont (iv)] $\text{span}_A(\sigma(B)) = B$.
\end{enumerate}
\end{proposition}
\begin{proof}
(ii) $\Rightarrow$ (i) and (iii) $\Rightarrow$ (iv) are obvious.

For (ii) $\Leftrightarrow$ (iii), note that $M^\sigma_b$ is just the change of basis matrix between the two bases $b_1, \ldots, b_r$ and $\sigma(b_1), \ldots, \sigma(b_r)$.

For (i) $\Rightarrow$ (ii), say $\sigma$ has invertible matrix with respect to $b_1, \ldots, b_r$, and let $\beta_1, \ldots, \beta_r$ be some other basis. Let $X$ be the change of basis matrix from the $b$ to the $\beta$, that is, $\beta_i = \sum_j x_{ji} b_j$, and $Y = X^{-1}$, and let $\mu_i$ be the $A$-module homomorphisms with $\mu_i(\sum_j a_j \beta_j) = a_i$. Then, 
\begin{align*}
    \sigma(\beta_i) &= \sum_j \sigma(x_{ji}) \sigma(b_j) \\
    &= \sum_j \sum_k \sigma(x_{ji}) \lambda_k(\sigma(b_j)) b_k \\
    &= \sum_j \sum_k \sum_n \sigma(x_{ji}) \lambda_k(\sigma(b_j)) y_{nk} \beta_n
\end{align*}
and so $\mu_n(\sigma(\beta_i)) = \sum_j \sum_k \sigma(x_{ji}) \lambda_k(\sigma(b_j)) y_{nk}$, that is, $M^\sigma_\beta = Y M^\sigma_b \sigma(X)$. Now since $X$ is invertible, $\sigma(X)$ is invertible in $\text{Mat}_{r \times r}(A)$. So $M^\sigma_\beta$ is invertible.

For (iv) $\Rightarrow$ (iii), assume $b_1, \ldots, b_r$ is an $A$-basis of $B$. Any $b \in B$ has $b = \sum_i a_i \sigma(\beta_i)$ for some $\beta_i \in B$. Also, $\beta_i = \sum_j \alpha_{ij} b_j$ since the $b_i$ are a basis, and so $b = \sum_i \sum_j a_i \sigma(\alpha_{ij}) \sigma(b_j)$. Then, $\sigma(b_1), \ldots, \sigma(b_r)$ spans $B$ over $A$. Now write $X$ for the matrix where $\sigma(b_i) = \sum_j x_{ji} b_j$, and $Y$ for the matrix where $b_i = \sum_j y_{ji} \sigma(b_j)$. Then, since $b_1, \ldots, b_r$ is a basis, we have that $XY = I$, and so by taking determinants, we see that $X$ and $Y$ are invertible in $\text{Mat}_{r \times r}(A)$. Then, $\sigma(b_1), \ldots, \sigma(b_r)$ is an $A$-basis of $B$.
\end{proof}

\begin{definition}
As a result of this proposition, having invertible matrix is independent of the choice of $A$-basis of $B$. We will say that $\sigma$ has invertible matrix if any of the above conditions hold.
\end{definition}

The following lemmas explain the connection between the endomorphism $\sigma$ having invertible matrix and being an automorphism.

\begin{lemma}\label{aut-iff-inv-matrix}
If $\sigma|_A \colon A \to A$ is an automorphism, then $\sigma$ is an automorphism on $B$ if and only if $\sigma$ has invertible matrix.
\end{lemma}
\begin{proof}
Define $B^\sigma$ to be the $A$-algebra with underlying ring $B$, but $A$-algebra structure map $a \mapsto \sigma(a)$. Since $\sigma|_A$ is an automorphism, $B^\sigma$ is a finite and free $A$-algebra of the same rank as $B$: in fact, if $b_1, \ldots, b_r$ is a basis of $B$, then it is also a basis of $B^\sigma$. Now the map $f \colon B \to B^\sigma$ given by $f(b) = \sigma(b)$ is actually $A$-linear, with
\begin{align*}
    f(b_i) &= \sum_j \lambda_j(\sigma(b_i)) b_j \\
    &= \sum_j \sigma ( \sigma|_A^{-1} \lambda_j (\sigma(b_i))) b_j
\end{align*}
and so the matrix of the $A$-linear map $f$ is $\sigma|_A^{-1}(M^\sigma_b)$. Then, $f$ is an isomorphism if and only if $\sigma|_A^{-1}(M^\sigma_b)$ is invertible, if and only if $\sigma$ has invertible matrix.
\end{proof}

\begin{lemma}
If $\sigma$ is an automorphism on $B$, then $\sigma|_A$ is an automorphism on $A$.
\end{lemma}
\begin{proof}
It is enough to show that $\sigma|_A$ is surjective onto $A$. Note that since $\sigma$ is surjective onto $B$, the $A$-linear span of $\{ \sigma(b_1), \ldots, \sigma(b_r) \}$ is $B$, and by a similar argument to the proof of (iv) $\Rightarrow$ (iii) in Proposition~\ref{inv-matrix-equivalences}, it must be an $A$-basis. Now, let $a \in A$. Then there is a $b \in B$ such that $a \sigma(b_1) = \sigma(b)$. Writing $b = \sum_{i=1}^r a_i b_i$ for some $a_i \in A$, we get $a \sigma(b_1) = \sum_{i=1}^r \sigma(a_i) \sigma(b_i)$. Since $\{ \sigma(b_1), \ldots, \sigma(b_r) \}$ is an $A$-basis, we get that $a = \sigma(a_1)$, and hence $\sigma|_A$ is surjective onto $A$.
\end{proof}

As a result, we see that if $\sigma$ is an automorphism of $B$, then it has invertible matrix. It turns out the converse is not true, as we point out in the following example.

\begin{example}
Let $A = \R(x_1, x_2, \ldots)$, $B = \C(x_1, x_2, \ldots)$, with basis $b_1 = 1, b_2 = i$, $\sigma|_\C = \id_\C$, and $\sigma(x_i) = x_{i+1}$. Note that $A$ and $B$ are fields and that $\sigma$ and $\sigma|_A$ are not surjective. However, the associated matrix is 
\[
M^\sigma_b = 
\renewcommand\arraystretch{1.3}
\mleft[
\begin{array}{c c}
  \lambda_1(\sigma(b_1))    & \lambda_1(\sigma(b_2)) \\
  \lambda_2(\sigma(b_1))    & \lambda_2(\sigma(b_2))
\end{array}
\mright]
=
\renewcommand\arraystretch{1.3}
\mleft[
\begin{array}{c c}
  1 & 0 \\
  0 & 1
\end{array}
\mright]
\]
which is invertible.
\end{example}

On the other hand, one can have an injective endomorphism $\sigma$ that does not have invertible matrix.

\begin{example}
Let $K$ be a field, $A = K[x]$ and $B = A[\varepsilon]/(\varepsilon^2)$ with $\sigma(p(x)+q(x) \varepsilon) = p(x)+xq(x) \varepsilon$. Then with respect to the basis $b = \{1, \varepsilon\}$, we have
\[
M^\sigma_b = 
\renewcommand\arraystretch{1.3}
\mleft[
\begin{array}{c c}
  \lambda_1(\sigma(b_1))    & \lambda_1(\sigma(b_2)) \\
  \lambda_2(\sigma(b_1))    & \lambda_2(\sigma(b_2))
\end{array}
\mright]
=
\renewcommand\arraystretch{1.3}
\mleft[
\begin{array}{c c}
  1 & 0 \\
  0 & x
\end{array}
\mright]
\]
which is not invertible in $\text{Mat}_{2 \times 2} (K[x])$.
\end{example}

\bigskip

\subsection{The matrix associated to a $\D$-ring}
We now extend the ideas of the previous subsection to the more general case of $\D$-rings. Just as we can associate a matrix to an endormophism of $B$, we can associate a matrix to a $\D$-ring structure on $B$ which, when invertible, will allow us to construct a left adjoint to $F^\D$ in Section~5. Here, we analyse this matrix and the conditions on its invertibility.

Let $(A,e)$ be a $\D$-ring and let $(B,f)$ be an $(A,e)$-algebra, where $B$ is a finite and free $A$-algebra. For any $k$-basis $\varepsilon_1, \ldots, \varepsilon_{l}$ of $\D$ and any $A$-basis $b_1, \ldots, b_r$ of $B$, consider the following $rl \times rl$ matrix with entries in $A$:
\[
M = 
\renewcommand\arraystretch{1.3}
\mleft[
\begin{array}{c c c c}
  M_{11}    & M_{12}    & \cdots & M_{1,l} \\
  M_{21}    & M_{22}    & \cdots & M_{2,l} \\
  \vdots    &           & \ddots &           \\
  M_{l,1} & M_{l,2} & \cdots & M_{l,l}
\end{array}
\mright]
\]
where $M_{mj}$ is the $r \times r$ matrix given by
\[
(M_{mj})_{ni} = \sum_{k=1}^{l} a_{jkm} \lambda_n(f_k(b_i))
\]

Recall that $\lambda_n \colon B \to A$ is the coordinate $b_n$-projection. The elements $a_{jkm} \in k$ are defined by $\varepsilon_j \varepsilon_k = \sum_{m=1}^{l} a_{jkm} \varepsilon_m$, and $f_k \colon B \to B$ is the coordinate of $f$ with respect to $\varepsilon_k$ given by $f(r) = \sum_{k=1}^{l} f_k(r) \varepsilon_k$. We call $M$ the \emph{matrix associated to $(B,f)$}.

We will now briefly explain where this matrix comes from and why we need to consider its invertibility. Define the functor $\D^e \colon \Alg{A} \to \Alg{A}$ by setting $\D^e(R)$ to be the ring $\D(R)$ but with $A$-algebra structure given by the composition of $e \colon A \to \D(A)$ with the natural map $\D(A) \to \D(R)$. See Notation~3.9 of \cite{moosa_scanlon_2010} and the discussion after Remark~2.10 of \cite{ozlem_2018}. On morphisms, $\D^e(\alpha) = \D(\alpha)$. We define $\D^f \colon \Alg{B} \to \Alg{B}$ similarly. Suppose $u \colon R \to \D(R)$ is a $\D$-ring structure on the $A$-algebra $R$. Then $(R,u)$ is an $(A,e)$-algebra if and only if $u$ is an $A$-algebra homomorphism $R \to \D^e(R)$.

We now define a natural transformation $\mu \colon F \D^e \to \D^f F$ in the following way: for any $A$-algebra $R$, we have a natural $A$-algebra homomorphism $\D^e(R) \to \D^f(R \otimes_A B)$ and an $A$-algebra homomorphism $B \to \D^f(R \otimes_A B)$ coming from the composition of $f$ with the natural map. Since $\D^e(R) \otimes_A B$ is the coproduct of $A$-algebras, we get an $A$-algebra homomorphism $\mu_R \colon \D^e(R) \otimes_A B \to \D^f(R \otimes_A B)$, which is also a $B$-algebra homomorphism. It is clear from its construction that $\mu$ is natural in $R$.

\begin{lemma}\label{lemma-M-is-matrix-of-mu}
The component of $\mu$ at $R$, $\mu_R \colon \D^e(R) \otimes_A B \to \D^f(R \otimes_A B)$, is an $R$-linear map of free $R$-modules with the natural $R$-module structure. With respect to the $R$-bases $\{ \varepsilon_n \otimes b_m \}$ of $\D^e(R) \otimes_A B$ and $\{ 1 \otimes b_n \varepsilon_m \}$ of $\D^f(R \otimes_A B)$, the matrix representation of $\mu_R$ is $M$. In particular, $\mu$ is a natural isomorphism if and only if $M$ is invertible. 
\end{lemma}
\begin{proof}
That $\mu_R$ is $R$-linear is clear from construction. The explicit formula for $\mu_R$ is given by
\begin{align*}
    \sum_{i=1}^r \left( \sum_{j=1}^l r_{ij} \varepsilon_j \right) \otimes b_i &\mapsto \sum_{i=1}^r \left( \sum_{j=1}^l r_{ij} \otimes 1 \ \varepsilon_j \right) \cdot \left( \sum_{k=1}^l 1 \otimes f_k(b_i) \varepsilon_k \right) \\
    &= \sum_{i=1}^r \sum_{j=1}^l \sum_{k=1}^l \sum_{m=1}^l a_{jkm} r_{ij} \otimes f_k(b_i) \ \varepsilon_m \\
    &= \sum_{i=1}^r \sum_{j=1}^l \sum_{k=1}^l \sum_{m=1}^l \sum_{n=1}^r  a_{jkm} \lambda_n(f_k(b_i)) r_{ij} \otimes b_n \ \varepsilon_m 
\end{align*}
which immediately shows that $M$ is the matrix of $\mu_R$ with respect to the aforementioned bases.
\end{proof}

From the lemma, we see that if $M$ is invertible, we have a natural transformation $W \D^f \to W \D^f FW \to W F \D^e W \to \D^e W$ coming from the composition of $\mu^{-1}$ with the unit and counit of the classical adjunction. If $g \colon C \to \D^f(C)$ is a $B$-algebra homomorphism, then composing the above natural transformation with the morphism $W(g) \colon W(C) \to W \D^f (C)$ gives an $A$-algebra homomorphism $g^W \colon W(C) \to \D^e W(C)$. In the next section, we will see that this $\D$-structure on $W(C)$ yields the left adjoint of $F^\D$. For now, we study the invertibility of $M$.

Note that $M$ depends on the choice of the $k$-basis of $\D$ and the $A$-basis of $B$. The following result shows us that invertibility of $M$ is actually independent of the $k$-basis of $\D$. After the proof of Theorem~\ref{M-inv-equiv-inv-sigma}, we will see that invertibility of $M$ is also independent of the $A$-basis of $B$.

\begin{proposition}\label{M-inv-independence}
Suppose we have two bases $\varepsilon = \{ \varepsilon_1, \ldots, \varepsilon_{l} \}$ and $\omega = \{ \omega_1, \ldots, \omega_{l} \}$ of $\D$, with $X$ the change of basis matrix from the $\varepsilon$ to the $\omega$, that is, $\omega_i = \sum_{j=1}^{l} x_{ji} \varepsilon_j$. Let $\tilde{X}$ be the $rl \times rl$ matrix obtained from $X$ by replacing each entry $x$ by the $r \times r$ block $xI$, where $I$ is the $r \times r$ identity matrix. Write $M^\varepsilon$ for the matrix $M$ corresponding to the basis $\varepsilon$ and similarly for $M^\omega$. Then,
\[
M^\omega = \tilde{X}^{-1} M^\varepsilon \tilde{X}
\]
\end{proposition}
\begin{proof}
Let $a_{ijk}$ be the product coefficients of the $\varepsilon$ and $\alpha_{ijk}$ for the $\omega$. Also, write $f^\varepsilon_i$ for the $i$th operator with respect to the basis $\varepsilon$ and similarly for $f^\omega_i$. We can obtain a relation between these by noting that the homomorphism $f \colon B \to \D(B)$ they induce must be the same, that is:
\[
\sum_{i=1}^{l} f_i^\varepsilon(b) \ \varepsilon_i = \sum_{i=1}^{l} f_i^\omega(b) \ \omega_i \text{ for all $b \in B$}
\]

To ease notation, let $\tilde{Y} = \tilde{X}^{-1}$. Let $N = \tilde{Y} M^\varepsilon \tilde{X}$. Then, the $mj$ block of $N$ is
\begin{align*}
N_{mj} &= \sum_p \sum_q \tilde{Y}_{mp} M^\varepsilon_{pq} \tilde{X}_{qj} \\
&= \sum_p \sum_q y_{mp} M^\varepsilon_{pq} x_{qj}
\end{align*}

Then the $ni$ element of $N_{mj}$ is 
\begin{align*}
    (N_{mj})_{ni} &= \sum_p \sum_q y_{mp} x_{qj} (M^\varepsilon_{pq})_{ni} \\
    &= \sum_p \sum_q \sum_k y_{mp} x_{qj} a_{qkp} \lambda_n(f^\varepsilon_k(b_i)) \\
    &= \sum_p \sum_q \sum_k y_{mp} x_{qj} a_{qkp} \lambda_n \left(\sum_r x_{kr} f^\omega_r(b_i) \right) \\
    &= \sum_p \sum_q \sum_k \sum_r x_{kr} y_{mp} x_{qj} a_{qkp} \lambda_n(f^\omega_r(b_i)) \\
    &= \sum_r \left( \sum_p \sum_q \sum_k x_{kr} y_{mp} x_{qj} a_{qkp} \right) \lambda_n(f^\omega_r(b_i))
\end{align*}

We now claim that $\alpha_{jrm} = \sum_{p,q,k} x_{kr} y_{mp} x_{qj} a_{qkp}$. Indeed, we have
\begin{align*}
    \omega_j \omega_r &= \left( \sum_q x_{qj} \varepsilon_q \right) \left( \sum_k x_{kr} \varepsilon_k \right) \\
    &= \sum_{q,k,p} x_{qj} x_{kr} a_{qkp} \varepsilon_p \\
    &= \sum_{q,k,p,u} x_{qj} x_{kr} a_{qkp} y_{up} \omega_u \\
\end{align*}
Then the claim follows.

It then follows that
\begin{align*}
    (N_{mj})_{ni} &= \sum_r \alpha_{jrm} \lambda_n(f^\omega_r(b_i)) \\
    &= \sum_k \alpha_{jkm} \lambda_n(f^\omega_k(b_i)) \\
    &= (M^\omega_{mj})_{ni}
\end{align*}
and hence $M^\omega = \tilde{Y} M^\varepsilon \tilde{X}$. 
\end{proof}

This proposition tells us that invertibility of $M$ is independent of which $k$-basis of $\D$ we choose. We now construct a ``suitable'' basis of $\D$ in order to characterise invertibility of $M$ in Theorem~\ref{M-inv-equiv-inv-sigma} below. This basis is constructed as follows. Write $\D = B_1 \times \ldots \times B_t$ where each $B_i$ is a local, finite-dimensional $k$-algebra with residue field $k$ (see Assumption~\ref{D-assumptions}). Let $\mathfrak{m}_i$ be the unique maximal ideal of $B_i$. Nakayama's Lemma tells us that $\mathfrak{m}_i$ is nilpotent: say $d_i$ is minimal such that $\mathfrak{m}_i^{d_i+1} = 0$. It then follows that for each $B_i$ we can find a $k$-basis $\mathcal{B}_i = \bigcup_{j = 0}^{d_i} \mathcal{B}_i^j$ where $\mathcal{B}_i^j$ is a $k$-basis of $\mathfrak{m}_i^j$ modulo $\mathfrak{m}_i^{j+1}$. Note that since the residue field of $B_i$ is $k$, we may choose $\mathcal{B}_i^0 = \{1\}$. Embed these bases inside $\D$ in the usual way, that is, if $x \in \mathcal{B}_i$, send $x$ to the element of $\D$ with $x$ in the $i$th position and zeros elsewhere. Then, the union of these forms a basis $\mathcal{B}$ of $\D$. Order $\mathcal{B} = \bigcup_{i=1}^t \bigcup_{j = 0}^{d_i} \mathcal{B}_i^j$ lexicographically on $i$ and $j$. The ordering of each $\mathcal{B}_i^j$ does not matter. We will write the elements of $\mathcal{B}$ as $\varepsilon_1, \ldots, \varepsilon_l$. Let $a_{jkm}$ be the product coefficients of $\mathcal{B}$; that is, $\varepsilon_j \varepsilon_k = \sum_{m=1}^l a_{jkm} \varepsilon_m$.

By the construction of the basis, we know that $\varepsilon_j \varepsilon_k = 0$ whenever $\varepsilon_j$ and $\varepsilon_k$ come from different $\mathcal{B}_i$. If they come from the same $\mathcal{B}_i$, then they can be expressed as a linear combination of $\mathcal{B}_i$, and so if $\varepsilon_m$ does not come from $\mathcal{B}_i$, it will not appear in this linear combination. So we see that $a_{jkm} = 0$ unless $\varepsilon_j$, $\varepsilon_k$, and $\varepsilon_m$ all come from the same $\mathcal{B}_i$.

Furthermore, if $\varepsilon_j \in \mathcal{B}_i^n$ and $\varepsilon_k \in \mathcal{B}_i^p$, then $\varepsilon_j \varepsilon_k \in \text{span}(\bigcup_{q=n+p}^{d_i} \mathcal{B}_i^q$). Hence, if $\varepsilon_m \in \mathcal{B}_i^q$ for $q < n+p$, $a_{jkm} = 0$. From these facts we can deduce the values of $a_{jkm}$ in specific cases:
\begin{enumerate}
    \item $m < j$ and $p>0$: $a_{jkm} = 0$.
    
    Since $m<j$, $q \leq n$, and hence $q < n+p$. By the above, $a_{jkm} = 0$.
    
    \item $m < j$ and $p=0$: $a_{jkm} = 0$.
    
    Since $p=0$, $\varepsilon_k$ is the $1$ in $B_i$. Then, $\varepsilon_j \varepsilon_k = \varepsilon_j \not = \varepsilon_m$.
    
    \item $m=j$ and $p>0$: $a_{jkm} = 0$.
    
    Again, as $m=j$, $q  = n$ and so $q < n+p$.
    
    \item $m=j$ and $p=0$: $a_{jkm} = 1$.
    
    $\varepsilon_j \varepsilon_k = \varepsilon_j = \varepsilon_m$. So $a_{jkm} = 1$.
\end{enumerate}

\bigskip

Now, recall the definition of the matrix $M$:
\[
M = 
\renewcommand\arraystretch{1.3}
\mleft[
\begin{array}{c c c c}
  M_{11}    & M_{12}    & \cdots & M_{1,l} \\
  M_{21}    & M_{22}    & \cdots & M_{2,l} \\
  \vdots    &           & \ddots &           \\
  M_{l,1} & M_{l,2} & \cdots & M_{l,l}
\end{array}
\mright]
\]
where
\[
(M_{mj})_{ni} = \sum_{k=1}^{l} a_{jkm} \lambda_n(f_k(b_i))
\]

With respect to the chosen basis, $\mathcal{B} = \{\varepsilon_1, \ldots, \varepsilon_l\}$, we now investigate the shape of each block $M_{mj}$ for $m \leq j$. Consider first the case when $m < j$. As pointed out above, if $\varepsilon_j$ and $\varepsilon_m$ belong to different $\mathcal{B}_i$, then $a_{jkm} = 0$ for all $k$. Otherwise, we are in cases (1) or (2) above, and hence $a_{jkm} = 0$ for all $k$. Hence, the block $M_{mj}$ is $0$.

Now consider the case $m = j$, that is, the block $M_{jj}$. Again, if $\varepsilon_j$ and $\varepsilon_m$ belong to different $B_i$, then $a_{jkj} = 0$ for all $k$. If they belong to the same $B_i$, then case (3) tells us that $a_{jkj} = 0$ when $p>0$, and (4) tells us that $a_{jkj} = 1$ when $p=0$. In conclusion, $(M_{jj})_{ni} = \lambda_n(f_k(b_i))$ where $k$ is such that $\varepsilon_k \in \mathcal{B}_r^0$ and $\varepsilon_j \in B_r$.

From Definition~\ref{associated-endomorphisms} we see that the $i$th projection map $\pi_i$ is just the map that projects onto the coefficient of $\varepsilon_k$ where $\varepsilon_k \in \mathcal{B}_i^0$. Hence, the $i$th associated endomorphism of $(B,f)$, denoted $\sigma_i$, is just $f_k$. Note that $\sigma_i$ has this form because of the chosen basis of $\D$.

So in all, $M$ is a block lower triangular matrix whose diagonal $r \times r$ blocks $M_{jj}$ are of the form
\[
M^{\sigma_i}_b = 
\mleft[
\begin{array}{c c c c}
  \lambda_1(\sigma_i(b_1)) & \lambda_1(\sigma_i(b_2)) & \cdots & \lambda_1(\sigma_i(b_r)) \\
  \lambda_2(\sigma_i(b_1)) & \lambda_2(\sigma_i(b_2)) & \cdots & \lambda_2(\sigma_i(b_r)) \\
  \vdots & & \ddots & \\
  \lambda_r(\sigma_i(b_1)) & \lambda_r(\sigma_i(b_2)) & \cdots & \lambda_r(\sigma_i(b_r))
\end{array}
\mright]
\]
for some $i = 1, \ldots, t$.

Note that $M^{\sigma_i}_b$ is the matrix associated to the endomorphism $\sigma_i$ as in Section~4.1. Hence, we have proved the following important result:

\begin{theorem}\label{M-inv-equiv-inv-sigma}
$M$ is invertible if and only if each associated endomorphism of $(B,f)$ has invertible matrix (in the sense of Section~4.1).
\end{theorem}

\begin{remark}\label{remark-independent-of-bases}
Combining this theorem with Propositions~\ref{inv-matrix-equivalences} and~\ref{M-inv-independence}, we see that invertibility of $M$ is independent of the choice of bases of $\D$ and $B$.
\end{remark}

\section{Weil Descent for $\D$-algebras}
In this section we prove the main theorem: Theorem~\ref{main-theorem} below. As before, we let $(A,e)$ be a $\D$-ring, $(B,f)$ an $(A,e)$-algebra where $B$ is a finite and free $A$-algebra.

The proofs in this section make use of the natural transformation $\mu \colon F \D^e \to \D^f F$ defined in the previous section whose invertibility is equivalent to the invertibility of the matrix $M$---the matrix associated to $(B,f)$---by Lemma~\ref{lemma-M-is-matrix-of-mu}. Furthermore, recall that in Theorem~\ref{M-inv-equiv-inv-sigma} we proved that $M$ is invertible if and only if the associated endomorphisms of $(B,f)$ have invertible matrix. For the remainder of this section, in addition to Assumption~\ref{D-assumptions}, we make the following assumption:

\begin{assumption}\label{ass=inv-matrix}
The associated endomorphisms of $(B,f)$ all have invertible matrix. Equivalently, $\mu$ is a natural isomorphism.
\end{assumption}

The following is part of the content of our main theorem.

\begin{theorem}
The $\D$-base change functor, $F^\D$, has a left adjoint $W^\D$. More precisely, for a $(B,f)$-algebra $(C,g)$, there exists a unique $\D$-structure $g^W$ on $W(C)$ that makes the unit of the classical adjunction, $\eta_C$, into a $\D$-homomorphism. $W^\D$ has the form $W^\D(C,g) = (W(C), g^W)$.
\end{theorem}

Before proving this result, we fix some notation. Since $W \dashv F$, we have the natural transformations given by the unit, $\eta \colon \id_{\Alg{B}} \to FW$, and the counit, $\varepsilon \colon WF \to \id_{\Alg{A}}$. We do not need to refer to the $k$-basis of $\D$ in this section, so we will use $\varepsilon$ to denote the counit. We will often identify a functor with the identity natural transformation on that functor. Recall the functors $\D^e \colon \Alg{A} \to \Alg{A}$ and $\D^f \colon \Alg{B} \to \Alg{B}$ defined in the previous section where $\D^e(R)$ is the ring $\D(R)$ but with $A$-algebra structure given by the composition of $e \colon A \to \D(A)$ with the natural map $\D(A) \to \D(R)$, and on morphisms, $\D^e(\alpha) = \D(\alpha)$. Recall also that a $\D$-ring structure on $R$ making it into an $(A,e)$-algebra is equivalent to an $A$-algebra homomorphism $R \to \D^e(R)$.

\begin{remark}\label{unique-str-on-tensor-is-F-comp-phi}
Suppose $(R,u)$ is an $(A,e)$-algebra so that $u \colon R \to \D^e(R)$ is an $A$-algebra homomorphism. Then $\mu_R \circ F(u) \colon F(R) \to \D^f F(R)$ is the $\D$-ring structure on $F(R)$ corresponding to $u \otimes f$ from the $\D$-base change functor in Section~3.2.
\end{remark}

We now use the natural isomorphism $\mu$ to define a suitable $\D$-ring structure on $W(C)$. For ease of notation we first define the natural transformation 
\begin{equation*}
    \zeta \colon W\D^f \to \D^e W
\end{equation*}
by the composition
\begin{equation*}\tag{$\ddagger$}
\begin{tikzcd}
    W \D^f \arrow[r, "W \D^f\eta"] & W \D^f FW \arrow[r, "W\mu^{-1} W"] & WF \D^e W \arrow[r, "\varepsilon \D^e W"] & \D^e W
\end{tikzcd}
\end{equation*}

Now suppose $(C,g)$ is a $(B,f)$-algebra, so that $g$ corresponds to the $B$-algebra homomorphism $g \colon C \to \D^f(C)$. Let $g^W = \zeta_C \circ W(g) \colon W(C) \to \D^e W(C)$. Then $(W(C), g^W)$ is an $(A,e)$-algebra. We define the functor $W^\D$ on objects as $W^\D(C,g) = (W(C),g^W)$ and on morphisms as $W^\D(\alpha) = W(\alpha)$. Since both $W$ and $\zeta$ are natural, it is clear that if $\alpha$ is a $\D$-homomorphism, then $W(\alpha)$ is a $\D$-homomorphism, so that $W^\D$ is actually a functor. We now need to show that $W^\D$ is left adjoint to $F^\D$ by showing that the natural bijection coming from the classical adjunction
\begin{align*}
     \text{\normalfont Hom}_{\Alg{A}}(W(C), R) & \to \text{\normalfont Hom}_{\Alg{B}}(C, F(R)) \\
    \phi & \mapsto F(\phi) \circ \eta_C \\
    \varepsilon_R \circ W(\psi) & \mapsfrom \psi
\end{align*}
restricts to a natural bijection 
\begin{align*}
     \text{\normalfont Hom}_{\Alg{(A,e)}}(W^\D(C,g), (R,u)) & \to \text{\normalfont Hom}_{\Alg{(B,f)}}((C,g), F^\D(R,u)) \\
    \phi & \mapsto F^\D(\phi) \circ \eta_C \\
    \varepsilon_R \circ W^\D(\psi) & \mapsfrom \psi
\end{align*}

We will do this by showing that both $\eta_C$ and $\varepsilon_R$ are $\D$-homomorphisms with the appropriate $\D$-structures defined above. Consider the following diagram of natural transformations:

\[
\begin{tikzcd}
FW \D^f \arrow[r, "FW \D^f\eta"] & FW \D^f FW \arrow[r, "FW\mu^{-1} W"] & FWF \D^e W \arrow[r, "F \varepsilon \D^e W"] & F \D^e W \arrow[r, "\mu W"] & \D^f FW\\
\D^f \arrow[u, "\eta \D^f"] \arrow[r, "\D^f\eta"] &\D^f FW \arrow[u, "\eta \D^f FW"] \arrow[r, "\mu^{-1} W"] & F \D^e W \arrow[u, "\eta F \D^e W"] \arrow[ru, equal]
\end{tikzcd}
\]

The squares commute due to naturality of $\eta$, and the equality is due to the adjunction axiom: $F \varepsilon \circ \eta F = F$. The composition along the top row is $\mu W \circ F\zeta$. By naturality of $\eta$, we get
\[
\begin{tikzcd}
    FW(C) \arrow[r, "FW(g)"] & FW \D^f(C) \\
    C \arrow[u, "\eta_C"] \arrow[r, "g"] & \D^f(C) \arrow[u, "\eta_{\D^f(C)}"]
\end{tikzcd}
\]
and putting these together we get
\[
\begin{tikzcd}
\D^f(C) \arrow[rr, "\D^f(\eta_C)"] & & \D^f FW(C) \\
C \arrow[rr, "\eta_C"] \arrow[u, "g"] & & FW(C) \arrow[u, "\mu_{W(C)} \circ F(g^W)", swap]
\end{tikzcd}
\]
so that $\eta_C$ is a $\D$-homomorphism by Remark~\ref{unique-str-on-tensor-is-F-comp-phi}.

Suppose now that $g'$ is a $\D$-ring structure $W(C) \to \D^e W(C)$ making $\eta_C$ into a $\D$-homomorphism, so that the following diagram of $B$-algebras commutes:
\begin{equation*}
\begin{tikzcd}
\D^f(C) \arrow[rr, "\D^f(\eta_C)"] & & \D^f FW(C) \\
C \arrow[rr, "\eta_C"] \arrow[u, "g"] & & FW(C) \arrow[u, "\mu_{W(C)} \circ F(g')", swap]
\end{tikzcd}
\end{equation*}

Since $\mu$ is an isomorphism, this is equivalent to the following diagram of $B$-algebras commuting:

\begin{equation}\tag{$*$}
\begin{tikzcd}
\D^f(C) \arrow[rr, "\mu^{-1}_{W(C)} \circ \D^f(\eta_C)"] & & F \D^e W(C) \\
C \arrow[rr, "\eta_C"] \arrow[u, "g"] & & FW(C) \arrow[u, "F(g')", swap]
\end{tikzcd}
\end{equation}

Consider now the diagram of $A$-algebras
\begin{equation*}
\begin{tikzcd}
W\D^f(C) \arrow[rr, "W(\mu^{-1}_{W(C)} \circ \D^f(\eta_C))"] & & WF \D^e W(C) \arrow[rr, "\varepsilon_{\D^e W(C)}"] & & \D^e W(C) \\
W(C) \arrow[rr, "W(\eta_C)"] \arrow[u, "W(g)"] & & WFW(C) \arrow[u, "WF(g')", swap] \arrow[rr, "\varepsilon_{W(C)}"] & & W(C) \arrow[u, "g'"]
\end{tikzcd}
\end{equation*}

Note that the left square commutes by applying $W$ to square $(*)$, and the right square commutes by naturality of $\varepsilon$. By the adjunction axiom $\varepsilon W \circ W \eta = W$, the composition along the bottom is $\id_{W(C)}$, and the composition along the top is $\zeta_C$ by definition. So $g^W = g'$, and we have proved the following.

\begin{lemma}\label{lemma-unique-D-str}
$g^W$ is the unique $\D$-structure on $W(C)$ making $(W(C), g^W)$ into an $(A,e)$-algebra and the unit, $\eta_C$, into a $\D$-homomorphism.
\end{lemma}

The adjunction axioms tell us that $F\varepsilon \circ \eta F  = F$, so that $\D^f F\varepsilon \circ \D^f\eta F = \D^f F$. Since $\mu$ is natural, the following diagram commutes :
\[
\begin{tikzcd}
F \D^e WF \arrow[r, "F\D^e\varepsilon"] \arrow[d, "\mu WF"]                                    & F \D^e \arrow[d, "\mu"]          \\
\D^f FWF \arrow[r, "\D^f F\varepsilon"]                             & \D^f F  \\
\D^f F \arrow[u, "\D^f\eta F"] \arrow[ru, equal] &                             
\end{tikzcd}
\]

Now apply $W$ and use naturality of the counit to get
\[
\begin{tikzcd}
\D^e WF \arrow[rr, "\D^e\varepsilon"] & & \D^e \\
WF \D^e WF \arrow[u, "\varepsilon \D^e WF"] \arrow[rr, "WF\D^e\varepsilon"] \arrow[d, "W\mu WF"]                                    & & WF \D^e \arrow[d, "W\mu"] \arrow[u, "\varepsilon \D^e"]         \\
W\D^f FWF \arrow[rr, "W\D^f F\varepsilon"]                             & & W\D^f F  \\
W\D^f F \arrow[u, "W\D^f\eta F"] \arrow[rru, equal] &  &                     
\end{tikzcd}
\]
Note that the composition up the left is precisely $\zeta F$. So $\varepsilon \D^e = \D^e\varepsilon \circ \zeta F \circ W\mu$.

Naturality of $\varepsilon$ gives
\[
\begin{tikzcd}
WF(R) \arrow[r, "WF(u)"] \arrow[d, "\varepsilon_R"] & WF \D^e(R) \arrow[d, "\varepsilon_{\D^e(R)}"] \arrow[rrr, "\zeta_{F(R)} \circ W(\mu_R)"] & & & \D^e WF(R) \arrow[dlll, "\D^e(\varepsilon_R)"] \\
R \arrow[r, "u"] & \D^e(R) &
\end{tikzcd}
\]
and since the composition along the top row is $(\mu_R \circ F(u))^W$, the counit $\varepsilon_R$ is a $\D$-homomorphism.

\begin{theorem}[The $\D$-Weil Descent]\label{main-theorem}
Suppose $(A,e)$ is a $\D$-ring and $(B,f)$ is an $(A,e)$-algebra, where $B$ is a finite and free $A$-algebra. Suppose also that the associated endomorphisms of $(B,f)$ all have invertible matrix. Then, the $\D$-base change functor, $F^\D \colon \Alg{(A,e)} \to \Alg{(B,f)}$ has a left adjoint denoted $W^\D$ called the $\D$-Weil descent. More precisely, $W^\D(C,g) = (W(C), g^W)$ where $g^W$ is the $\D$-ring structure defined by $\zeta_C \circ W(g)$ and $\zeta \colon W \D^f \to \D^e W$ is the natural transformation defined in equation $(\ddagger)$.

In fact, the natural bijection $\tau(C,R)$ from the classical adjunction restricts to a natural bijection:
\[
\tau^\D((C,g),(R,u)) \colon \text{\normalfont Hom}_{\Alg{(A,e)}}(W^\D(C,g), (R,u)) \to \text{\normalfont Hom}_{\Alg{(B,f)}}((C,g), F^\D(R,u))
\]
\end{theorem}

\begin{remark}
If we apply this theorem to the case when $\D = k$, we get what we call the difference Weil descent and denote it $W^\sigma$. In this case, $\D$-rings are rings with a single (not necessarily injective) endomorphism.
\end{remark}

\section{Further remarks}

In this section we investigate three further aspects. Firstly, we make some observations about properties of the associated endomorphisms that are transferred by the $\D$-Weil descent. In particular, we prove that if an associated endomorphism of $(C,g)$ is trivial, then the same is true of the $\D$-Weil descent, $(W(C), g^W)$. Secondly, we prove results about the composition of a $\D_1$-structure and a $\D_2$-structure and their Weil descents. In particular, we will show that commutativity of these structures is preserved after taking the Weil descent. These two subsections imply that the result of this paper is an actual generalisation of the case of several commuting derivations from \cite{differential_weil_descent}. Thirdly, we explore the necessity of the condition that the associated endomorphisms of $(B,f)$ have invertible matrix for the existence of the $\D$-Weil descent.

Throughout this section, unless stated otherwise, $(A,e)$ is a $\D$-ring, $(B,f)$ is an $(A,e)$-algebra, where $B$ is finite and free over $A$, and $(C,g)$ is a $(B,f)$-algebra. Assumption~\ref{D-assumptions} is still in force.

\bigskip

\subsection{Transfer properties of the associated endomorphisms}

Recall from Definition~\ref{associated-endomorphisms} the projection maps for $\D$. If $\D = \prod_{i=1}^t B_i$ where each $B_i$ is a local $k$-algebra with residue field $k$, then $\pi_i \colon \D \to B_i \to k$ is the composition of the projection onto the $i$th component of $\D$ with the residue map onto $k$. These $\pi_i$ lift to $R$-algebra homomorphisms $\pi_i^R \colon \D(R) \to R$. Then the associated endomorphisms of a $\D$-ring $(R,e)$ are defined by $\pi_i^R \circ e$ for each $i=1, \ldots, t$.

\begin{lemma}\label{lemma-ends-of-Weil-descent}
Let $(C,g)$ be a $(B,f)$-algebra, and suppose that the associated endomorphisms of $(B,f)$ have invertible matrix. Then the associated endomorphisms of the $\D$-Weil descent of $(C,g)$ are the difference Weil descents of the associated endomorphisms of $(C,g)$. In particular, if an associated endomorphism of $(C,g)$ is trivial, then so is the corresponding one of $W^\D(C,g)$.
\end{lemma}
\begin{proof}
Let $(\sigma_i), (\tau_i), (\upsilon_i), (\rho_i)$ be the associated endomorphisms of $(A,e)$, $(B,f)$, $(C,g)$, $(W(C), g^W)$, respectively. We need to show that $\rho_i = \upsilon_i^{W^\sigma}$. Consider the following diagrams for each $i=1, \ldots, t$:
\[
\begin{tikzcd}
C \arrow[r, "\eta_C"] & F(W(C))  \\
\D(C) \arrow[u, "\pi^C_i"] \arrow[r, "\D(\eta_C)"] & \D(F(W(C))) \arrow[u, "\pi^{F(W(C))}_i", swap] \\
C \arrow[u, "g"] \arrow[r, "\eta_C"] & F(W(C)) \arrow[u, "g^W \otimes f", swap]
\end{tikzcd}
\]

The compositions of the vertical maps on the left are $\upsilon_i$ by definition. On the right they are $\rho_i \otimes \tau_i$ by Lemma~\ref{lemma-ass-ends-of-tensor-prod}. Hence $\rho_i$ is a difference structure on $W(C)$ that makes $(W(C), \rho_i)$ into an $(A, \sigma_i)$-algebra and $\eta_C$ into a $(B, \tau_i)$-algebra homomorphism. Since $\tau_i$ has invertible matrix, Lemma~\ref{lemma-unique-D-str} tells us that such a difference structure is unique, and so we must have $\rho_i = \upsilon_i^{W^\sigma}$. 

For the in particular clause, since the following square commutes
\[
\begin{tikzcd}
C \arrow[r, "\eta_C"] & F(W(C)) \\
C \arrow[u, "\id_C"] \arrow[r, "\eta_C"] & F(W(C)) \arrow[u, "\id_{W(C)} \otimes \id_B", swap]
\end{tikzcd}
\]
we must have $(\id_C)^W = \id_{W(C)}$ by the uniqueness of the difference structure on $W(C)$ making it an $(A,\id_A)$-algebra and $\eta_C$ a $(B,\id_B)$-algebra homomorphism. Note here that $\id_B$ has invertible matrix.
\end{proof}

\begin{remark}
This lemma tells us that we may apply our $\D$-Weil descent result (Theorem~\ref{main-theorem}) in the context of \cite{moosa_scanlon_2013}. Recall that, there, the authors impose the condition that for a $\D$-ring $(R,e)$, $e$ must also be a section to $\pi_1^R$, and hence that the first associated endomorphism must be the identity. Thus, if $(A,e)$ and $(B,f)$ have trivial first associated endomorphism, we may consider the category of $(A,e)$-algebras $(R,u)$ where $u$ has trivial first associated endomorphism, and similarly for $(B,f)$-algebras. Denote these subcategories $\Alg{(A,e)}^*$ and $\Alg{(B,f)}^*$. One checks that $F^\D$ can now be considered as a functor $\Alg{(A,e)}^* \to \Alg{(B,f)}^*$, and the previous lemma tells us that $W^\D$ restricts to a functor $\Alg{(B,f)}^* \to \Alg{(A,e)}^*$ which is still left adjoint to $F^\D$. In particular, our result is an actual generalisation of the single derivation case from \cite{differential_weil_descent}, since the category of differential $A$-algebras is equal to $\Alg{(A,e)}^*$ when we take $\D = k[\varepsilon]/(\varepsilon^2)$.
\end{remark}

We point out here that $W^\D$ does not in general preserve injectivity of the associated endomorphisms. That is, if the $i$th associated endomorphism of $g$ is injective, the $i$th associated endomorphism of $g^W$ may no longer be injective.

\begin{example}\label{example-inj-not-preserved}
Let $\D = k$ so that the associated endomorphism of a $\D$-ring structure is just the $\D$-ring structure itself. Let $A = \mathbb{F}_2$ be the field with two elements, and let $B = \mathbb{F}_2 [\varepsilon] / (\varepsilon^2)$. Let $\id_A$ and $\id_B$ be the $\D$-ring structures on $A$ and $B$ respectively. Note then that if $(C,\rho)$ is a $(B,\id_B)$-algebra, $\rho$ is a $B$-algebra endomorphism of $C$ making the following diagram commute:
\[
\begin{tikzcd}
C \arrow[r, "\eta_C"] & F(W(C)) \\
C \arrow[u, "\rho"] \arrow[r, "\eta_C"] & F(W(C)) \arrow[u, "\rho^W \otimes \id_B", swap]
\end{tikzcd}
\]

Note also that since $\rho$ is a $B$-algebra endomorphism, it is a morphism in $\Alg{B}$, and so we may apply the classical Weil descent to it. Theorem~\ref{adjunction-conditions} tells us that $W(\rho) = \rho^W$.

Let $C = B[t]$ and let $\rho$ be the unique map extending $\id_B$ on $B$ and sending $t \mapsto t^2$. Then $\rho$ is injective. Recall from Section~2 that $W(B[t]) = A[t] \otimes_A A[t]$ and that $\eta_C(t) = t(1) \otimes 1 + t(2) \otimes \varepsilon$. Then

\begin{align*}
    \eta_C (\rho (t)) &= \eta_C (t^2) \\
    &= \eta_C(t)^2 \\
    &= t(1)^2 \otimes 1
\end{align*}
where the last equality holds because $\varepsilon^2 = 0$ and we are in characteristic $2$.

Also
\begin{align*}
    (W(\rho) \otimes \id_B) (\eta_C(t)) &= W(\rho)(t(1)) \otimes 1 + W(\rho)(t(2)) \otimes \varepsilon
\end{align*}

By the commutativity of the diagram, we have that $W(\rho)(t(1)) = t(1)^2$ and $W(\rho)(t(2)) = 0$. Hence $W(\rho)$ is not injective.
\end{example}

\begin{remark}
\begin{enumerate}
    \item This example tells us that in general the difference Weil descent functor does not restrict to the categories of algebras equipped with an injective endomorphism. However, Corollary~\ref{difference-monoid-hom} will tell us that the difference Weil descent does preserve automorphisms, and hence will restrict to a functor in the categories of inversive difference algebras (see \cite{levin_difference_2008}).
    \item The example above uses in an essential way the fact that the characteristic is positive. We are not currently aware of such an example in characteristic zero.
\end{enumerate}
\end{remark}

\bigskip

\subsection{The composition of a $\D_1$-structure and a $\D_2$-structure}

Suppose we now have two finite-dimensional $k$-algebras $\D_1 = \prod_{i=1}^{t_1}B_i$ and $\D_2 = \prod_{j=1}^{t_2} C_j$ where each $B_i$ and $C_j$ are local with residue field $k$. Then, $\D_2 \otimes_k \D_1 = \prod_{i=1}^{t_1} \prod_{j=1}^{t_2} C_j \otimes_k B_i$. From \cite{sweedler_1975} we know that $C_j \otimes_k B_i$ is local with residue field $k$, and hence $\D_2 \otimes_k \D_1$ satisfies Assumption~\ref{D-assumptions}. We may then consider the category of $\D_2 \otimes_k \D_1$-rings. We will write these as $\D_1 \D_2$-rings since
\[
(\D_2 \otimes_k \D_1)(R) = R \otimes_k (\D_2 \otimes_k \D_1) \cong (R \otimes_k \D_2) \otimes_k \D_1 = \D_1(\D_2(R))
\]
for a $k$-algebra $R$.

If some $k$-algebra $R$ has a $\D_1$-structure $e_1$ and a $\D_2$-structure $e_2$, we can form a $\D_1 \D_2$-structure on $R$ by the $k$-algebra homomorphism
\[
\D_1(e_2) \circ e_1 \colon R \to \D_1 \D_2(R)
\]
We now investigate the Weil descent of this composition of $\D_1$-structures and $\D_2$-structures. Suppose $R$, $S$, and $T$ all have a $\D_1$-structure $e_1, f_1, g_1$ and a $\D_2$-structure $e_2,f_2,g_2$ that make $(S,f_1)$ and $(T,g_1)$ into $(R,e_1)$-algebras and $(S,f_2)$ and $(T,g_2)$ into $(R,e_2)$-algebras. We can then define $\D_1 \D_2$-structures on each of them as above.

\begin{lemma}\label{tensor-prod-of-composition}
Assuming the notation of the paragraph above, we have
\[
\D_1(f_2 \otimes g_2) \circ (f_1 \otimes g_1) = (\D_1(f_2) \circ f_1) \otimes (\D_1(g_2) \circ g_1).
\]
\end{lemma}
\begin{proof}
 Consider the following diagram
\[
\begin{tikzcd}
                                                                        & \mathcal{D}_1 \mathcal{D}_2(S) \arrow[rr]                    &                                                              & \mathcal{D}_1 \mathcal{D}_2(S \otimes_R T)                                 \\
\mathcal{D}_1 \mathcal{D}_2(R) \arrow[rr] \arrow[ru]                    &                                                              & \mathcal{D}_1 \mathcal{D}_2(T) \arrow[ru]                    &                                                                            \\
                                                                        & \mathcal{D}_1(S) \arrow[rr] \arrow[uu, "\mathcal{D}_1(f_2)", pos=0.3] &                                                              & \mathcal{D}_1(S \otimes_R T) \arrow[uu, "\mathcal{D}_1(f_2 \otimes g_2)"'] \\
\mathcal{D}_1(R) \arrow[ru] \arrow[rr] \arrow[uu, "\mathcal{D}_1(e_2)"] &                                                              & \mathcal{D}_1(T) \arrow[ru] \arrow[uu, "\mathcal{D}_1(g_2)", pos=0.3] &                                                                            \\
                                                                        & S \arrow[rr] \arrow[uu, "f_1", pos=0.3]                               &                                                              & S \otimes_R T \arrow[uu, "f_1 \otimes g_1"']                               \\
R \arrow[ru] \arrow[rr] \arrow[uu, "e_1"]                               &                                                              & T \arrow[ru] \arrow[uu, "g_1", pos=0.3]                               &                                                                           
\end{tikzcd}
\]

The horizontal maps are just the natural maps. The lower cube commutes due to the definition of the tensor product of $\D_1$-algebras, and the upper cube commutes by applying $\D_1$ to the cube that commutes due to the definition of the tensor product of $\D_2$-algebras. This means that $\D_1(f_2 \otimes g_2) \circ (f_1 \otimes g_1)$ is a $\D_1 \D_2$-structure on $S \otimes_R T$ that extends the ones on $S$ and $T$, and hence by uniqueness of the tensor product of $\D_1 \D_2$-structures, we must have that
\[
\D_1(f_2 \otimes g_2) \circ (f_1 \otimes g_1) = (\D_1(f_2) \circ f_1) \otimes (\D_1(g_2) \circ g_1)
\]
\end{proof}

We now return to the case when $B$ is a finite and free $A$-algebra, and let $C$ be a $B$-algebra.

\begin{definition}
Let $\DStr_B(C)$ be the collection of triples $(e,f,g)$ where $e$ is a $\D$-structure on $A$, $f$ one on $B$, and $g$ one on $C$ such that $(B,f)$ is an $(A,e)$-algebra and $(C,g)$ is a $(B,f)$-algebra, and the associated endomorphisms of $(B,f)$ have invertible matrix. For any $A$-algebra $R$, let $\DStr_A(R)$ be the collection of pairs $(e,u)$ where $e$ is a $\D$-structure on $A$ and $u$ one on $R$ such that $(R,u)$ is an $(A,e)$-algebra. The $\D$-Weil descent then tells us that we have a map
\begin{align*}
    (\,\cdot\,)^{W^\D} \colon \DStr_B(C) & \to \DStr_A(W(C)) \\
    (e,f,g) & \mapsto (e,g^{W^\D})
\end{align*}

Unless we need to be precise, we will drop the tuple notation and just use $g$ for $(e,f,g)$ and $u$ for $(e,u)$. We will also suppress the $\D$ notation in the map $(\,\cdot\,)^{W^\D}$ and just write $(\,\cdot\,)^W$. In what follows, we will make use of these maps for $\D_1$, $\D_2$, and $\D_1 \D_2$, but it will be clear from context which we mean: $(\,\cdot\,)^{W^{\D_1}}$ will be applied only to $\D_1$-structures, $(\,\cdot\,)^{W^{\D_2}}$ only to $\D_2$-structures, and $(\,\cdot\,)^{W^{\D_1 \D_2}}$ only to $\D_1 \D_2$-structures.
\end{definition}

\begin{lemma}\label{lemma-theta-map}
The following map is well-defined.
\begin{align*}
    \Theta_B \colon \DoneStr_B(C) \times \DtwoStr_B(C) & \to \DonetwoStr_B(C) \\
    ((e_1,f_1,g_1), (e_2,f_2,g_2)) & \mapsto (\D_1(e_2) \circ e_1, \D_1(f_2) \circ f_1, \D_1(g_2) \circ g_1)
\end{align*}
\end{lemma}
\begin{proof}
Since $(e_1,f_1,g_1) \in \DoneStr_B(C)$, the following diagram commutes:
\[
\begin{tikzcd}
    \D_1(A) \arrow[r] & \D_1(B) \arrow[r] & \D_1(C) \\
    A \arrow[u, "e_1"] \arrow[r] & B \arrow[u, "f_1"] \arrow[r] & C \arrow[u, "g_1"]
\end{tikzcd}
\]

Since $(e_2,f_2,g_2) \in \DtwoStr_B(C)$, we get a similar diagram. Apply $\D_1$ to this second diagram and compose the vertical maps to get the following commuting diagram:
\[
\begin{tikzcd}
    \D_1 \D_2(A) \arrow[r] & \D_1 \D_2(B) \arrow[r] & \D_1 \D_2(C) \\
    A \arrow[u, "\D_1(e_2) \circ e_1"] \arrow[r] & B \arrow[u, "\D_1(f_2) \circ f_1"] \arrow[r] & C \arrow[u, "\D_1(g_2) \circ g_1)"]
\end{tikzcd}
\]
So these $\D_1 \D_2$-structures make the algebra structure maps into $\D_1 \D_2$-homomorphisms.

Finally, we need to check that the associated endomorphisms of $(B,\D_1(f_2) \circ f_1)$ have invertible matrix. Recall that the associated endomorphisms are defined using the projection maps $\D_2 \otimes_k \D_1 \to k$. For $1 \leq i \leq t_1$ and $1 \leq j \leq t_2$, we will say that the $(i,j)$th projection map for $\D_2 \otimes_k \D_1$ is the composition $\D_2 \otimes_k \D_1 \to C_j \otimes_k B_i \to k$. Then we claim that the $(i,j)$th associated endomorphism of $(B,\D_1(f_2) \circ f_1)$ is $\sigma_j \tau_i$ where the $\tau_i$ are the associated endomorphisms of $(B,f_1)$ and the $\sigma_j$ are those of $(B,f_2)$. To see this, consider the following commuting diagram:

\[
\begin{tikzcd}
\mathcal{D}_1 \mathcal{D}_2(B) \arrow[rr, "\mathcal{D}_1(\pi^2_j)"]                                           & & \mathcal{D}_1(B) \arrow[rr, "\pi^1_i"] & & B \\
\\
\mathcal{D}_1(B) \arrow[rr, "\pi^1_i"] \arrow[uu, "\mathcal{D}_1(f_2)"] \arrow[rruu, "\mathcal{D}_1(\sigma_j)"'] & & B \arrow[rruu, "\sigma_j"']             & &   \\
\\
B \arrow[uu, "f_1"] \arrow[rruu, "\tau_i"']                                                                     &                                       &  & &
\end{tikzcd}
\]
where $\pi_i^1$ is the $i$th projection map for $\D_1$ and $\pi_j^2$ is the $j$th projection map for $\D_2$.

The lower triangle commutes due to the definition of $\tau_i$. The triangle in the upper left commutes by applying $\D_1$ to the definition of $\sigma_j$. It remains to show that the composition along the top row is the $(i,j)$th projection map for $\D_2 \otimes_k \D_1$. But this follows from the commutativity of the following diagram
\[
\begin{tikzcd}
\D_2 \otimes_k \D_1 \arrow[r] \arrow[dr] & C_j \otimes_k \D_1 \arrow[r] \arrow[d] & k \otimes_k \D_1 \arrow[d] \\
& C_j \otimes_k B_i \arrow[r] \arrow[dr] & k \otimes_k B_i \arrow[d] \\
& & k \otimes_k k
\end{tikzcd}
\]
where the composition along the top row is $\D_1(\pi^2_j)$, the composition along the right column is $\pi^1_i$ and the diagonal composition is the $(i,j)$th projection map for $\D_2 \otimes_k \D_1$.

Recall that Proposition~\ref{inv-matrix-equivalences} says that an endomorphism has invertible matrix if and only if it sends any $A$-basis of $B$ to another $A$-basis. Then, since $\tau_i$ and $\sigma_j$ both have invertible matrix, $\sigma_j \tau_i$ must as well.
\end{proof}

A similar proof also shows that we have a well-defined map
\begin{align*}
    \Theta_A \colon \DoneStr_A(R) \times \DtwoStr_A(R) & \to \DonetwoStr_A(R) \\
    ((e_1,u_1),(e_2,u_2)) & \mapsto (\D_1(e_2) \circ e_1, \D_1(u_2) \circ u_1)
\end{align*}

We also get maps
\[
\DtwoStr_B(C) \times \DoneStr_B(C) \to \DtwooneStr_B(C)
\]
and
\[
\DtwoStr_A(R) \times \DoneStr_A(R) \to \DtwooneStr_A(R)
\]
by exchanging the roles of $\D_1$ and $\D_2$. We will also denote these maps by $\Theta_B$ and $\Theta_A$, but it will be clear from context which one we mean.

\begin{theorem}\label{theorem-Theta-compatible-with-W}
For $g_1 \in \DoneStr_B(C)$ and $g_2 \in \DtwoStr_B(C)$,
\[
\Theta_B(g_1,g_2)^W = \Theta_A(g_1^W, g_2^W)
\]
\end{theorem}
\begin{proof}
Using the $\D_1$-Weil descent and the $\D_2$-Weil descent, the following squares commute:
\[
\begin{tikzcd}
    \D_1(C) \arrow[r, "\D_1(\eta_C)"] & \D_1(FW(C)) \\
    C \arrow[u, "g_1"] \arrow[r, "\eta_C"] & FW(C) \arrow[u, "g_1^W \otimes f_1", swap]
\end{tikzcd}
\]
\[
\begin{tikzcd}
    \D_2(C) \arrow[r, "\D_2(\eta_C)"] & \D_2(FW(C)) \\
    C \arrow[u, "g_2"] \arrow[r, "\eta_C"] & FW(C) \arrow[u, "g_2^W \otimes f_2", swap]
\end{tikzcd}
\]

Apply $\D_1$ to the second square and compose the vertical maps so that the following square commutes:
\[
\begin{tikzcd}
    \D_1 \D_2(C) \arrow[rr, "\D_1 \D_2(\eta_C)"] & & \D_1 \D_2(FW(C)) \\
    C \arrow[u, "\D_1(g_2) \circ g_1"] \arrow[rr, "\eta_C"] & & FW(C) \arrow[u, "\D_1(g_2^W \otimes f_2) \circ (g_1^W \otimes f_1)", swap]
\end{tikzcd}
\]

By Lemma~\ref{tensor-prod-of-composition}, the right vertical map is equal to $(\D_1(g_2^W) \circ g_1^W) \otimes (\D_1(f_2) \circ f_1)$. And hence, by the uniqueness of the $\D_1 \D_2$-structure on $W(C)$ that makes it into an $(A, \D_1(e_2) \circ e_1)$-algebra and $\eta_C$ into a $(B, \D_1(f_2) \circ f_1)$-algebra homomorphism, we must have
\[
(\D_1(g_2) \circ g_1)^W = \D_1(g_2^W) \circ g_1^W
\]
\end{proof}

We now apply this theorem to the difference case. Let $\D = \D_1 = \D_2 = k$. Then, $\D_2 \otimes_k \D_1 = k$ and $\Theta_A$ and $\Theta_B$ are just composition of endomorphisms. $\DStr_B(C)$ is a monoid with composition $\Theta_B$ and identity $(\id_A, \id_B, \id_C)$. Similary, $\DStr_A(R)$ is a monoid under $\Theta_A$ and $(\id_A, \id_R)$.

\begin{corollary}\label{difference-monoid-hom}
In the notation of the above paragraph,
\[
(\,\cdot\,)^W \colon \DStr_B(C) \to \DStr_A(W(C))
\]
is a monoid homomorphism.
\end{corollary}
\begin{proof}
We have that $\Theta_B(g_1, g_2)^W = (g_1 \circ g_2)^W$ and $\Theta_A(g_1^W, g_2^W) = g_1^W \circ g_2^W$. Then Theorem~\ref{theorem-Theta-compatible-with-W} tells us that $(g_1 \circ g_2)^W = g_1^W \circ g_2^W$. Lemma~\ref{lemma-ends-of-Weil-descent} then tells us that $(\id_C)^W = \id_{W(C)}$.
\end{proof}

\begin{remark}
Corollary~\ref{difference-monoid-hom} tells us that the difference Weil descent restricts to the categories of inversive difference algebras, that is, algebras equipped with an automorphism. Indeed, if $(A,e)$, $(B,f)$ and $(C,g)$ are all inversive difference algebras, applying $(\,\cdot\,)^W$ to the equations $g \circ g^{-1} = \id_C = g^{-1} \circ g$ tells us that $g^W$ is also an automorphism on $W(C)$.
\end{remark}

We now further develop these results to study the commutativity of a $\D_1$-structure and a $\D_2$-structure. Let $\Gamma$ be the canonical isomorphism
\begin{align*}
    \Gamma \colon \D_2 \otimes_k D_1 & \to \D_1 \otimes_k \D_2 \\
    \alpha_2 \otimes \alpha_1 & \mapsto\alpha_1 \otimes \alpha_2
\end{align*}
For any $k$-algebra $S$, $\Gamma$ lifts to $\Gamma^S \colon S \otimes_k \D_2 \otimes_k \D_1 \to S \otimes_k \D_1 \otimes_k \D_2$ in the usual way. Therefore, $\Gamma$ induces maps $\DonetwoStr_B(C) \to \DtwooneStr_B(C)$ and $\DonetwoStr_A(R) \to \DtwooneStr_A(R)$ by applying the appropriate $\Gamma$ coordinate-wise. We will also denote these maps $\Gamma$. It should be clear from context which we mean.

\begin{definition}
Let $S$ be a $k$-algebra, equipped with a $\D_1$-structure $e_1$ and a $\D_2$-structure $e_2$. We will say that $e_1$ commutes with $e_2$ if
\[
\Gamma^S \circ \D_1(e_2) \circ e_1 = \D_2(e_1) \circ e_2.
\]

For $g_1 \in \DoneStr_B(C)$ and $g_2 \in \DtwoStr_B(C)$, we will say that $g_1$ commutes with $g_2$ if
\[
\Gamma \circ \Theta_B(g_1, g_2) = \Theta_B(g_2, g_1).
\]
Similarly, for $u_1 \in \DoneStr_A(R)$ and $u_2 \in \DtwoStr_A(R)$, we will say that $u_1$ commutes with $u_2$ if $\Gamma \circ \Theta_A(u_1, u_2) = \Theta_A(u_2, u_1)$.
\end{definition}

\begin{remark}
If we choose bases of $\D_1$ and $\D_2$ and think of $e_1$ and $e_2$ as their corresponding sequence of free operators, the condition
\[
\Gamma^S \circ \D_1(e_2) \circ e_1 = \D_2(e_1) \circ e_2
\]
says that every operator of $e_1$ commutes with every operator of $e_2$.
\end{remark}

We now prove a modification of Theorem~\ref{theorem-Theta-compatible-with-W} that includes $\Gamma$.

\begin{lemma}\label{lemma-Gamma-respects-W}
For $g_1 \in \DoneStr_B(C)$ and $g_2 \in \DtwoStr_B(C)$,
\[
(\Gamma \circ \Theta_B(g_1, g_2))^W = \Gamma \circ \Theta_A(g_1^W, g_2^W)
\]
\end{lemma}
\begin{proof}
Firstly, suppose $e_1, f_1, g_1$ and $e_2, f_2, g_2$ are $\D_1$- and $\D_2$-structures on $R$, $S$, and $T$ making $(S,f_1)$ and $(T,g_1)$ into $(R,e_1)$-algebras and $(S,f_2)$ and $(T,g_2)$ into $(R,e_2)$-algebras. Consider the following diagram:
\[
\begin{tikzcd}
                                                                & \mathcal{D}_2 \mathcal{D}_1(S) \arrow[rr]            &                                                      & \mathcal{D}_2 \mathcal{D}_1(S \otimes_R T)            \\
\mathcal{D}_2 \mathcal{D}_1(R) \arrow[rr] \arrow[ru]            &                                                      & \mathcal{D}_2 \mathcal{D}_1(T) \arrow[ru]            &                                                       \\
                                                                & \mathcal{D}_1 \mathcal{D}_2(S) \arrow[rr] \arrow[uu, "\Gamma^S", pos=0.3] &                                                      & \mathcal{D}_1 \mathcal{D}_2(S \otimes_R T) \arrow[uu, "\Gamma^{S \otimes_R T}"] \\
\mathcal{D}_1 \mathcal{D}_2(R) \arrow[ru] \arrow[rr] \arrow[uu, "\Gamma^R"] &                                                      & \mathcal{D}_1 \mathcal{D}_2(T) \arrow[ru] \arrow[uu, "\Gamma^T", pos=0.3] &                                                       \\
                                                                & S \arrow[rr] \arrow[uu]                              &                                                      & S \otimes_R T \arrow[uu]                              \\
R \arrow[ru] \arrow[rr] \arrow[uu]                              &                                                      & T \arrow[ru] \arrow[uu]                              &                                                      
\end{tikzcd}
\]
where the horizontal maps are the usual ones and the vertical ones in the lower cube are the compositions of the $\D_1$-structure and $\D_2$-structure. By the uniqueness of the $\D_2 \D_1$-structure on $S \otimes_R T$, we have that
\[
\Gamma^{S \otimes_R T} \circ \D_1(f_2 \otimes g_2) \circ (f_1 \otimes g_1) = \left( \Gamma^S \circ \D_1(f_2) \circ f_1 \right) \otimes \left( \Gamma^T \circ \D_1(g_2) \circ g_1 \right)
\]

Now, returning to the original context, the following diagram also commutes:
\[
\begin{tikzcd}
\D_2 \D_1(C) \arrow[r] & \D_2 \D_1(F(W(C))) \\
\D_1 \D_2(C) \arrow[r] \arrow[u, "\Gamma^C"] & \D_1 \D_2(F(W(C))) \arrow[u, "\Gamma^{F(W(C))}", swap] \\
C \arrow[r, "\eta_C"] \arrow[u, "\D_1(g_2) \circ g_1"] & F(W(C)) \arrow[u, "(\D_1(g_2^W) \circ g_1^W) \otimes (\D_1(f_2) \circ f_1)", swap]
\end{tikzcd}
\]

Hence, $\Gamma^{W(C)} \circ \D_1(g_2^W) \circ g_1^W$ is a $\D_2 \D_1$-structure on $W(C)$ making it into an $(A, \Gamma^A \circ \D_1(e_2) \circ e_1)$-algebra and $\eta_C$ into a $(B, \Gamma^B \circ \D_1(f_2) \circ f_1)$-algebra homomorphism. If the associated endomorphisms of $\Gamma^B \circ \D_1(f_2) \circ f_1$ all had invertible matrix, then by the uniqueness of such a $\D_2 \D_1$-structure, we must have that $(\Gamma^C \circ \D_1(g_2) \circ g_1)^W = \Gamma^{W(C)} \circ \D_1(g_2^W) \circ g_1^W$, from which the result follows.

Now, note that the $(i,j)$th projection map for $\D_1 \otimes_k \D_2 = \prod_{i=1}^{t_2} \prod_{j=1}^{t_1} B_j \otimes C_i$ is $\D_1 \otimes_k \D_2 \to B_j \otimes_k C_i \to k$. Let $\pi_{(i,j)}^{\D_1 \otimes \D_2}$ denote the $(i,j)$th projection map for $\D_1 \otimes_k \D_2$, and let $\pi_{(i,j)}^{\D_2 \otimes \D_1}$ denote the $(i,j)$th projection map for $\D_2 \otimes_k \D_1$: $\D_2 \otimes_k \D_1 \to C_j \otimes_k B_i \to k$. Then $\pi_{(i,j)}^{\D_1 \otimes \D_2} \circ \Gamma = \pi_{(j,i)}^{\D_2 \otimes \D_1}$. Thus, the $(i,j)$th associated endomorphism of $\Gamma^B \circ \D_1(f_2) \circ f_1$ is the $(j,i)$th associated endomorphism of $\D_1(f_2) \circ f_1$, $\sigma_i \tau_j$, which has invertible matrix.
\end{proof}

\begin{corollary}\label{commuting-D1-D2}
Let $g_1 \in \DoneStr_B(C)$ and $g_2 \in \DtwoStr_B(C)$. If $g_1$ commutes with $g_2$, then $g_1^W$ commutes with $g_2^W$.
\end{corollary}
\begin{proof}
If $g_1$ commutes with $g_2$, then $\Gamma \circ \Theta_B(g_1, g_2) = \Theta_B(g_2, g_1)$. Applying $(\,\cdot\,)^W$ to this equation and using Theorem~\ref{theorem-Theta-compatible-with-W} and Lemma~\ref{lemma-Gamma-respects-W}, we get $\Gamma \circ \Theta_A(g_1^W, g_2^W) = \Theta_A(g_2^W, g_1^W)$. Hence, $g_1^W$ commutes with $g_2^W$.
\end{proof}

For a $k$-algebra $S$, we will say that a $\D$-structure $e$ on $S$ commutes if $\Gamma^S \circ \D(e) \circ e = \D(e) \circ e$. Note that this is equivalent to saying that, with respect to a fixed basis of $\D$, the free operators corresponding to $e$ pairwise commute. For $g \in \DStr_B(C)$, we will say $g$ commutes if $\Gamma \circ \Theta_B(g,g) = \Theta_B(g,g)$, and similarly for $u \in \DStr_A(R)$, $u$ commutes if $\Gamma \circ \Theta_A(u,u) = \Theta_A(u,u)$. An immediate consequence of Corollary~\ref{commuting-D1-D2} is the following.

\begin{corollary}\label{commuting-D}
Let $g \in \DStr_B(C)$. If $g$ commutes, then $g^W$ commutes.
\end{corollary}

These results allow us to deduce that commutativity is preserved by the $\D$-Weil descent in several cases. We give details for the case of $m$ endomorphisms and $n$ derivations.

\begin{example}
Suppose $\D = k[x_1, \ldots, x_n] / (x_1, \ldots, x_n)^2 \times k^m$ and that for every $\D$-structure, the first associated endomorphism is trivial (unless $n=0$, in which case we do not impose that any associated endomorphism is trivial). Then, a $\D$-structure is a collection of $n$ derivations and $m$ endomorphisms. Suppose also that for a given $A$, $B$, $C$, all of the derivations and endomorphisms pairwise commute. Then, by Corollary~\ref{commuting-D}, we have that the Weil descents of all the derivations and endomorphisms pairwise commute.
\end{example}

\begin{remark}
The $n=0$ case also follows from Corollary~\ref{difference-monoid-hom}. The $m=0$ case appears in \cite{differential_weil_descent}.
\end{remark}

\begin{remark}
It seems possible that Corollary~\ref{commuting-D} could be extended to a more general context where commutativity is replaced by an iterativity condition as in Section~2.2 of \cite{moosa_scanlon_ghs}. We leave this for future work as it goes beyond the scope of this paper.
\end{remark}

\bigskip

\subsection{On the necessity of having invertible matrix}

It is a natural question to ask whether a converse to our main theorem holds.

\begin{question}
If $F^\D$ has a left adjoint, must every associated endomorphism of $(B,f)$ have invertible matrix?
\end{question}

We do not yet know the answer in general, but we do have the following partial converse which imposes some mild conditions on such a left adjoint. We use the following notation. For each $z \in \D(B)$, let $g^z \colon B[t] \to \D(B[t])$ be the $\D$-structure on $B[t]$ that extends $f$ on $B$ and sends $t \mapsto z$. 

\begin{theorem}\label{partial-converse}
Suppose $F^\D$ has a left adjoint, $W^\D$, and that for each $z \in \D(B)$ the underlying $A$-algebra of $W^\D(B[t], g^z)$ is a faithfully flat $A$-module. Then, the associated endomorphisms of $(B,f)$ all have invertible matrix.
\end{theorem}
\begin{proof}
Note that by \cite[Section~1.9]{van_den_dries_bounds_1984}, for any $R$-algebra $S$, $S$ is a faithfully flat $R$-module if and only if $S$ is a flat $R$-module and every linear system of equations defined over $R$ which has a solution in $S$ already has a solution in $R$.

For $z \in \D(B)$, consider the unit of the adjunction $\eta \colon (B[t], g^z) \to F^\D W^\D(B[t], g^z)$. That this is a $\D$-homomorphism at $t$ means that 
\begin{align*}\tag{$*$}
    \sum_{m=1}^{l} \sum_{n=1}^r \lambda_n(\eta(g^z_m(t))) \otimes b_n \ & \varepsilon_m = \\ 
    \sum_{i=1}^r \sum_{j=1}^{l} &\sum_{k=1}^{l} \sum_{n=1}^r \sum_{m=1}^{l} a_{jkm} \lambda_n(f_k(b_i)) h^z_j(\lambda_i(\eta(t))) \otimes b_n \ \varepsilon_m
\end{align*}
where $h^z$ is the $\D$-structure on $W^\D(B[t],g^z)$. Write $z = \sum_m \beta_m \varepsilon_m$ and $\beta_m = \sum_n a_{nm} b_n$. Then, $\lambda_n ( \eta (g^z_m(t))) = \lambda_n (\eta(\beta_m)) = \lambda_n (\beta_m) = a_{nm}$ since $\eta$ is a $B$-algebra homomorphism.

Let $\bar{a}$ be the vector in $A^{rl}$ of the elements $a_{nm}$. Then, equation $(*)$ tells us that we have a solution in $W^\D(B[t],g^z)$ to the system $\bar{a} = M \bar{x}$. Since $W^\D(B[t],g^z)$ is faithfully flat, we have a solution in $A$, and hence $M$ is onto as a linear map $A^{rl} \to A^{rl}$. Let $e_i$ be the standard basis of $A^{rl}$ and $u_i \in A^{rl}$ with $M u_i = e_i$. Then, the matrix whose columns are $u_i$ is a right inverse to $M$. Taking determinants tells us that $M$ is invertible.
\end{proof}

If $A$ is a field, then $W^\D(B[t], g^z)$ is a free $A$-module---hence faithfully flat---so Theorems~\ref{main-theorem} and \ref{partial-converse} yield the following:

\begin{corollary}\label{corollary-inv-matrix-is-necessary}
Suppose $A$ is a field. Then, $F^\D$ has a left adjoint if and only if the associated endomorphisms of $(B,f)$ all have invertible matrix.
\end{corollary}

This result specialises to the difference case: 

\begin{corollary}
Suppose $(K,\sigma) \leq (L,\tau)$ is an extension of difference fields where $L/K$ is finite and $\sigma$ is an automorphism. Then, the difference base change functor has a left adjoint (the difference Weil descent). 
\end{corollary}
\begin{proof}
Note that by Lemma~\ref{aut-iff-inv-matrix}, $\tau$ is an automorphism if and only if it has invertible matrix. Since $\sigma$ is an automorphism, $L/K$ is a finite-dimensional $K$-vector space, and $\tau$ is injective, $\tau$ must also be an automorphism.
\end{proof}

\section{Appendix: An explicit construction of the $\D$-Weil descent}

While the construction of the $\D$-ring structure $g^W$ given in Section~5 is very natural, it does not yield an explicit or computational construction. In this appendix we will sketch a construction that parallels the classical one. Let $\varepsilon_1, \ldots, \varepsilon_l$ be a $k$-basis of $\D$ and $b_1, \ldots, b_r$ an $A$-basis of $B$. We continue to impose Assumptions~\ref{D-assumptions} and \ref{ass=inv-matrix}.

We need some notation and a technical result relating the matrix $M$ to whether an algebra homomorphism is also a $\D$-homomorphism.

\begin{notation}
For a collection of elements $\{x_{ij} \colon 1 \leq i \leq r, 1 \leq j \leq l \}$ in some $A$-algebra, we write $(x_{ij})$ for the $rl$-vector ordered reverse lexicographically on the indices $i$ and $j$. We write $M \cdot (x_{ij})$ to denote the standard matrix multiplication of an $rl \times rl$ matrix with an $rl$-vector. Thus, the result is an $rl$-vector.
\end{notation}

\begin{lemma}\label{technical-lemma}
Let $(C,g)$ be a $(B,f)$-algebra, $(R,u)$ an $(A,e)$-algebra, and $\phi \colon C \to F(R) = R \otimes_A B$ a $B$-algebra homomorphism. Then, $\phi$ is a $(B,f)$-algebra homomorphism if and only if the following equation holds for every $c \in C$:
\[\tag{$*$}
(\lambda_i \phi g_j(c)) = M \cdot (u_j\lambda_i \phi(c))
\]

As a result, when $M$ is invertible, the values $u_j \lambda_i \phi(c)$ are uniquely determined.
\end{lemma}
\begin{proof}
$\phi$ is a $(B,f)$-algebra homomorphism if and only if it is a $\D$-homomorphism, if and only if the following diagram commutes:
\[
\begin{tikzcd}
\D(C) \arrow[r, "\D(\phi)"] & \D(F(R)) \\
C \arrow[u, "g"] \arrow[r, "\phi"] & F(R) \arrow[u, "u \otimes f"]
\end{tikzcd}
\]

Now expand both compositions and equate coefficients of the $b_n \varepsilon_m$.
\end{proof}

\begin{remark}\label{remark-generators-M}
If some $S \subseteq C$ generates $C$ as a $B$-algebra, then it is enough to ask for equality $(*)$ to hold for every $s \in S$.
\end{remark}

Our explicit construction of the $\D$-Weil descent parallels the classical construction. So we need the algebraic notions of $\D$-ideals, $\D$-quotients, and $\D$-polynomial rings. The proofs will be mostly omitted---they all essentially follow from simple computations using the fact that $\D$ is a free and finite $k$-algebra.

\begin{definition}
Let $(R,e)$ be a $\D$-ring, and let $I$ be an ideal of $R$. We say that $I$ is a $\mathcal{D}$-ideal if $e(I) \subseteq \mathcal{D}(I) \vcentcolon = I \otimes_k \D$. Note that $\D(I)$ is an ideal of $\D(R)$: if $IR \subseteq I$, then 
\[
\D(I) \cdot \D(R) = (I \otimes_k \D) \cdot (R \otimes_k \D) \subseteq I \otimes_k \D.
\]
\end{definition}

\begin{remark}
In the context of Example~\ref{D-ring-examples}(1) where $\sigma$ is trivial, $I$ is a $\D$-ideal if and only if it is a differential ideal, that is, if $\delta(I) \subseteq I$. For Example~\ref{D-ring-examples}(2), $\D$-ideals are ideals with $\sigma_i(I) \subseteq I$ for each $i$.
\end{remark}

\begin{lemma}
Let $(R,e)$ and $(S,f)$ be two $\mathcal{D}$-rings and suppose $\phi \colon R \to S$ is a $\mathcal{D}$-homomorphism. Then, $\ker \phi$ is a $\mathcal{D}$-ideal.
\end{lemma}

\begin{lemma}\label{D-quotients}
Let $I$ be an ideal of $R$. Then, $\mathcal{D}(R) / \mathcal{D}(I) \cong \mathcal{D}(R/I)$.
\end{lemma}

\begin{lemma}\label{D-structure-on-quotient}
Let $(R,e)$ be a $\mathcal{D}$-ring, $(S,f)$ an $(R,e)$-algebra, and $I$ a $\mathcal{D}$-ideal of $S$. Then, there exists a unique $\mathcal{D}$-ring structure on the quotient $S / I$ given by
\begin{align*}
    \bar{f} \colon S / I & \to \mathcal{D}(S / I) \\
    s + I & \to f(s) + \mathcal{D}(I)
\end{align*}
which makes the quotient map $q \colon S \to S/I$ into an $(R,e)$-algebra homomorphism.
\end{lemma}

We now need the natural notion of a $\D$-polynomial ring. These have been defined in Section~3.1 of \cite{moosa_scanlon_ghs} (implicitly) and in Remark~3.8 of \cite{moosa_scanlon_2013}. We expand on the details here. First fix a $k$-basis $\varepsilon = \{ \varepsilon_1, \ldots, \varepsilon_{l} \}$ of $\D$. For any $k$-algebra $R$, $\D(R)$ has $R$-basis $\{ 1~\otimes~\varepsilon_1 , \ldots , 1~\otimes~\varepsilon_l \}$. As before, we will abuse notation and write $\varepsilon_i$ for $1 \otimes \varepsilon_i$ in $\D(R)$, but it will be clear from context which ring we are working in. If $(R,e)$ is a $\D$-ring, we denote by $e_i \colon R \to R$ the coordinate maps of $e$ with respect to the basis $\varepsilon$. That is, the maps $e_i$ are the additive operators of $R$ such that $e(r) = \sum_{i=1}^l e_i(r) \varepsilon_i$ for all $r \in R$.

\begin{definition}\label{def-theta-etheta}
We denote by $\Theta$ the set of all finite words on the alphabet $\{ 1, \ldots, l \}$. For a $\D$-ring $(R,e)$ and $\theta \in \Theta$, we will write $e_\theta$ for the corresponding composition of coordinatised $\D$-operators. For example, if $\theta = 123$, then $e_\theta = e_3 \circ e_2 \circ e_1$. Note then that $e_{\theta_1 \theta_2} = e_{\theta_2} \circ e_{\theta_1}$.
\end{definition}

\begin{definition}\label{D-polynomial-ring-def}
Let $(R,e)$ be a $\D$-ring and $T = (t)_{t \in T}$ a collection of indeterminates. The $\D$-polynomial algebra in indeterminates $T$ over $(R,e)$ with respect to $\varepsilon$ is the ring 
\[
R\{T\}_\D^\varepsilon = R[t^\theta \colon t \in T \text{ and } \theta \in \Theta]
\]
where $(t^\theta)_{t \in T, \theta \in \Theta}$ is a new family of indeterminates, equipped with homomorphism 
\begin{align*}
    e' \colon R\{T\}_\D^\varepsilon & \to \D(R\{T\}_\D^\varepsilon) \\
    t^\theta & \mapsto t^{\theta 1} \varepsilon_1 + t^{\theta 2} \varepsilon_2 + \ldots t^{\theta l} \varepsilon_l \\
    r & \mapsto e(r)
\end{align*}
This makes $(R\{T\}_\D^\varepsilon, e')$ an $(R,e)$-algebra.
\end{definition}

Suppose $(S,f)$ is an $(R,e)$-algebra and $X \subseteq S$. We denote by $R\{X\}_\D$ the $\D$-ring generated in $S$ by $X$ over $(R,e)$. This is a well-defined notion since the intersection of a collection of $\D$-subrings is a $\D$-subring.

\begin{lemma}\label{D-poly-onto}
Suppose that $(S,f)$ is an $(R,e)$-algebra which is generated as a $\D$-ring by the (possibly infinite) tuple $\bar{a} = (a_i)_{i \in I}$ over $(R,e)$, so that $S = R\{\bar{a}\}_\D$. Let $\bar{t} = (t_i)_{i \in I}$ be a tuple of indeterminates. Then, there exists a unique, surjective $(R,e)$-algebra homomorphism $\text{\normalfont ev}_{\bar{a}} \colon R\{\bar{t}\}_\D^\varepsilon \to S$ which maps $t_i \mapsto a_i$ for each $i \in I$.
\end{lemma}
\begin{proof}
Define $\text{ev}_{\bar{a}}(t_i^\theta) = f_\theta (a_i)$ (see Definition~\ref{def-theta-etheta}). Then, $\text{ev}_{\bar{a}}$ is clearly a surjective $R$-algebra homomorphism. To show it is a $\D$-homomorphism, we need to show that the following diagram commutes:
\[
\begin{tikzcd}
\mathcal{D}(R\{\bar{t}\}_\mathcal{D}^\varepsilon) \arrow[r, "\mathcal{D}(\text{ev}_{\bar{a}})"] & \mathcal{D}(S)   \\
R\{\bar{t}\}_\mathcal{D}^\varepsilon \arrow[u, "e"] \arrow[r, "\text{ev}_{\bar{a}}"]            & S \arrow[u, "f"]
\end{tikzcd}
\]
This follows from a short computation.
\end{proof}

This lemma yields:

\begin{corollary}
Suppose that $\varepsilon$ and $\omega$ are two bases of $\D$. Then, $R\{T\}_\D^\varepsilon$ and $R\{T\}_\D^\omega$ are isomorphic as $(R,e)$-algebras.
\end{corollary}

As a result of this corollary, we omit the superscript and just write $R\{T\}_\D$.

\begin{remark}
Combining the above lemmas, we see that any $(R,e)$-algebra is a quotient of some $\D$-polynomial algebra over $(R,e)$ by a $\D$-ideal.
\end{remark}

We now return to the construction of the $\D$-Weil restriction. As usual, $(A,e)$ is a $\D$-ring, and $(B,f)$ is an $(A,e)$-algebra where $B$ is finite and free as an $A$-module with basis $b_1, \ldots, b_r$. Recall that the component of the unit of the classical adjunction at the polynomial algebra $B\{T\}_\D$ is
\[
\eta_{B\{T\}_\D} \left( t^\theta \right) = \sum_{i=1}^r t^\theta(i) \otimes b_i
\]
We first construct the $\D$-Weil descent for a $\D$-polynomial algebra over $(B,f)$.

\begin{lemma}\label{lemma-D-str-on-poly-ring}
Let $T$ be a set of indeterminates. Then, there exists a $\D$-structure $s$ on $W(B\{T\}_\D)$ making $(W(B\{T\}_\D), s)$ into an $(A,e)$-algebra and $\eta_{B\{T\}_\D}$ into a $\D$-homomorphism.
\end{lemma}
\begin{proof}
$W(B\{T\}_\D) = A\{T\}_\D^{\otimes r}$, and applying Lemma~\ref{technical-lemma} with $\eta_{B\{T\}_\D}$ tells us that $\eta_{B\{T\}_\D}$ is a $\D$-homomorphism if and only if 
\[
(\lambda_i \eta_{B\{T\}_\D} h_j (t^\theta)) = M \cdot (s_j \lambda_i \eta_{B\{T\}_\D}(t^\theta))
\]
for every $t^\theta$. Now, $(s_j \lambda_i \eta_{B\{T\}_\D}(t^\theta)) = s_j(t^\theta(i))$, and $(\lambda_i \eta_{B\{T\}_\D} h_j (t^\theta)) = t^{\theta j}(i)$. Since $M$ is invertible we have $s_j(t^\theta(i)) = M^{-1} \cdot (t^{\theta j}(i))$. This gives an explicit expression for $s_j$ on each generator of $A\{T\}^{\otimes r}_\D$ and hence an explicit expression for $s$. Since $A\{T\}^{\otimes r}_\D$ is a polynomial algebra, this gives a $\D$-ring structure on $A\{T\}^{\otimes r}_\D$ making $\eta_{B\{T\}_\D}$ into a $\D$-homomorphism.
\end{proof}

\begin{remark}
Note that $W(B\{T\}_\D)$ is a polynomial algebra, but that, in general, $s$ is not the $\D$-ring structure that makes $(W(B\{T\}_\D), s)$ a $\D$-polynomial algebra as in Definition~\ref{D-polynomial-ring-def}---it is twisted by $M^{-1}$. The same occurs in the differential case; see the proof of Theorem 3.2 of \cite{differential_weil_descent}.
\end{remark}

Now let $(C,g)$ be a $(B,f)$-algebra. By Lemma~\ref{D-poly-onto}, there is a set of indeterminates $T$ and a surjective $(B,f)$-algebra homomorphism $\pi_C \colon B\{T\}_\D \to C$, where $B\{T\}_\D$ has the standard $\D$-structure $h$ extending $f$ with $h(t^\theta) = t^{\theta 1} \varepsilon_1 + \ldots + t^{\theta l} \varepsilon_{l}$. The component of the unit of the classical adjunction at $C$, $\eta_C$, is given by

\[
\eta_C \left( \pi_C (t^{\theta}) \right) = \sum_{i=1}^r W(\pi_C)(t^{\theta}(i)) \otimes b_i
\]

The $\D$-ring structure from Lemma~\ref{lemma-D-str-on-poly-ring} will induce one on $W(C) = W(B\{T\}_\D) / I_C$. Recall from Section~2.1 the definition of the ideal $I_C$. This ideal is generated by the elements $\lambda_i (\eta_{B\{T\}_\D}(\gamma))$ as $\gamma$ ranges over $\ker \pi_C$, and $W(\pi_C)$ is the residue map of this ideal.

\begin{lemma}\label{lemma-Ic-is-D-ideal}
The ideal $I_C$ of $W(B\{T\}_\D)$ is a $\D$-ideal for the $s$ given in Lemma~\ref{lemma-D-str-on-poly-ring}.
\end{lemma}
\begin{proof}
Let $\gamma \in \ker \pi_C$. By definition of $I_C$, we need to show $s(\lambda_i(\eta_{B\{T\}_\D}(\gamma))) \in \D(I_C)$ for each $i$, that is, that the vector $(s_j \lambda_i \eta_{B\{T\}_\D} (\gamma)) \in I_C$.

Since $\eta_{B\{T\}_\D}$ is a $\D$-homomorphism, we have
\[
(\lambda_i \eta_{B\{T\}_\D} h_j(\gamma)) = M \cdot (s_j \lambda_i \eta_{B\{T\}_\D} (\gamma))
\]

Now $\ker \pi_C$ is a $\D$-ideal for $h \colon B\{T\}_D \to \D(B\{T\}_D)$ (the standard $\D$-polynomial structure), and so $h_j(\gamma) \in \ker \pi_C$. Then, by construction of $I_C$, $(\lambda_i \eta_{B\{T\}_\D} h_j(\gamma))$ is in $I_C$. Since $M$ is invertible, $(s_j \lambda_i \eta_{B\{T\}_\D} (\gamma)) \in I_C$.
\end{proof}

Lemmas~\ref{D-structure-on-quotient} and \ref{lemma-Ic-is-D-ideal} imply that the $s$ from Lemma~\ref{lemma-D-str-on-poly-ring} induces a $\D$-structure $g^W$ on $W(C) = W(B\{T\}_\D) / I_C$ which makes it an $(A,e)$-algebra and $W(\pi_C)$ an $(A,e)$-algebra homomorphism. One then readily checks it makes $\eta_C$ into a $\D$-homomorphism by an argument similar to Theorem~3.2 of \cite{differential_weil_descent}.

Therefore, we have provided an explicit way to construct the $\D$-Weil descent $W^\D(C,g) = (W(C), g^W)$.

\end{document}